\documentclass{amsart}
\usepackage{amssymb}
\usepackage[v2,cmtip]{xy}

\DeclareMathAlphabet{\scr}{U}{rsfs}{m}{n}

\def\draftdate{October 10, 2013}

\newcommand{\noloc}{\;{:}}

\newcommand{\EnRing}[1][{n}]{\mathop{\aE_{#1}\aR\mskip-1.75mu\it ing}\nolimits}
\newcommand{\HoRing}{\mathop{\it Ho\aR\mskip-1.75mu\it ing}\nolimits}
\newcommand{\Top}{\mathop{\aT\mskip-5.5mu\it op}\nolimits}
\newcommand{\EnSpace}[1][{n}]{\mathop{\aE_{#1}\aT\mskip-5.5mu\it op}\nolimits}
\newcommand{\HSpace}{\mathop{\aH\aT\mskip-5.5mu\it op}\nolimits}
\newcommand{\Mod}{\mathop{\aM\mskip-4.0mu\it od}\nolimits}
\newcommand{\Modho}[1][{R}]{\Mod^{E_{n}}_{\Ho R}}
\newcommand{\Modps}[1][{R}]{\Mod^{E_{n}}_{R}}
\newcommand{\Spectra}{\aM}
\newcommand{\Stable}{\aS}

\newcommand{\LMS}{\aS\mskip-7.5mu_{\text{\tiny  LMS}}}
\newcommand{\EKMM}{\aM_{S}}

\newcommand{\Ass}{\mathop{\oA\it ss}\nolimits}

\newcommand{\rO}{\mathrm{O}}
\newcommand{\rCn}{\mathrm{C}_{n}}

\newcommand{\rd}{\mathbf{R}}

\newcommand{\six}[1][{6}]{\langle #1\rangle}

\newcommand{\tofrom}{\xymatrix@C-.5pc{\ar@<.5ex>[r]&\ar@<.5ex>[l]}}

\newcommand{\bC}{{\mathbb{C}}}
\newcommand{\bL}{{\mathbb{L}}}
\newcommand{\bN}{{\mathbb{N}}}
\newcommand{\bP}{{\mathbb{P}}}
\newcommand{\bU}{{\mathbb{U}}}
\newcommand{\bZ}{{\mathbb{Z}}}

\DeclareMathAlphabet{\mathscr}{U}{rsfs}{m}{n}
\let\catsymbfont\mathscr

\newcommand{\aE}{{\catsymbfont{E}}}
\newcommand{\aF}{{\catsymbfont{F}}}
\newcommand{\aI}{{\catsymbfont{I}}}
\newcommand{\aH}{{\catsymbfont{H}}}

\newcommand{\aM}{{\catsymbfont{M}}}

\newcommand{\aR}{{\catsymbfont{R}}}
\newcommand{\aS}{{\catsymbfont{S}}}
\newcommand{\aT}{{\catsymbfont{T}}}

\let\opsymbfont\mathcal

\newcommand{\oA}{{\opsymbfont{A}}}
\newcommand{\oC}{{\opsymbfont{C}}}
\newcommand{\oL}{{\opsymbfont{L}}}
\newcommand{\oO}{{\opsymbfont{O}}}
\newcommand{\oR}{{\opsymbfont{R}}}

\newcommand{\iso}{\cong}     
\newcommand{\sma}{\wedge}    

\renewcommand{\to}{\mathchoice{\longrightarrow}{\rightarrow}{\rightarrow}{\rightarrow}}
\newcommand{\from}{\mathchoice{\longleftarrow}{\leftarrow}{\leftarrow}{\leftarrow}}
\newcommand{\overto}[1]{\xrightarrow{\,#1\,}}

\def\quickop#1{\expandafter\DeclareMathOperator\csname #1\endcsname{#1}}
\quickop{Id}\quickop{id}\quickop{Ho}

\quickop{colim}

\newtheorem{thm}[equation]{Theorem}

\newtheorem{cor}[equation]{Corollary}
\newtheorem{lem}[equation]{Lemma}
\newtheorem{prop}[equation]{Proposition}

\theoremstyle{definition}
\newtheorem{defn}[equation]{Definition}
\newtheorem{cons}[equation]{Construction}

\theoremstyle{remark}
\newtheorem{rem}[equation]{Remark}

\numberwithin{equation}{section}

\newcommand{\term}[1]{\emph{#1}}

\long\def\ignore#1\endignore{\relax}

\begin{document}

\title{$E_{n}$ Genera}
\author{Steven Greg Chadwick}
\address{Department of Mathematics, University of California,
Riverside, CA}
\email{steven.chadwick@ucr.edu}
\author{Michael A. Mandell}
\address{Department of Mathematics, Indiana University, Bloomington, IN}
\email{mmandell@indiana.edu}
\thanks{The second author was supported in part by NSF grant DMS-1105255}

\date{\draftdate}


\begin{abstract}
Let $R$ be an $E_{2}$ ring spectrum with zero odd dimensional homotopy
groups.  Every map of ring spectra $MU\to R$ is represented by a map
of $E_{2}$ ring spectra.  If $2$ is invertible in $\pi_{0}R$, then
every map of ring spectra $MSO\to R$ is represented by a map of
$E_{2}$ ring spectra.
\end{abstract}

\maketitle

\section{Introduction}

Genera (in the sense we use the word here) are multiplicative
cobordism invariants of manifolds with extra structure.  In the past
60 years, the study of various genera has led to stunning advances
throughout mathematics, from algebraic geometry with the
Hirzebruch-Riemann-Roch theorem
\cite{Hirzebruch-RiemannRochOrig,Hirzebruch-RiemannRoch}, to
differential equations with the Atiyah-Singer index theorem
\cite{Palais-AtiyahSinger}, to mathematical physics with the Witten
genus \cite{Witten-Wittengenus}, in addition to innumerable advances
inside topology.
Because our perspective comes from stable homotopy theory, we will
restrict attention to genera that extend to singular manifolds on
pairs.  With only minor additional hypotheses, such genera are
precisely natural transformations of cohomology theories, or better,
maps of ring spectra out of a cobordism spectrum or a related
spectrum.  These genera lie at the heart of modern stable homotopy
theory, in particular, its organization in terms of chromatic
phenomena, which derives from Quillen's identification of genera of
stably almost complex manifolds (i.e., ring spectrum maps out of $MU$)
in terms of formal coordinates for formal group laws.

The three most basic cobordism spectra $MO$ (unoriented cobordism),
$MSO$ (oriented cobordism), and $MU$ (complex cobordism) are all
examples of $E_{\infty}$ 
ring spectra (now usually called commutative $S$-algebras): These are
ring spectra where the multiplication is not just associative,
commutative, and unital in the stable category, but actually in a
point-set symmetric monoidal category of spectra.  The $E_{\infty}$
structures on these cobordism spectra derive from products and powers
of manifolds, and work of Ando, Hopkins, Rezk, and Strickland (and
their collaborators, among others) shows that refining maps out of
cobordism spectra and related spectra to $E_{\infty}$ (or $H_{\infty}$)
ring maps has implications in geometry as well as topology and stable
homotopy theory (see, for example, \cite{AHR, AndoThesis,AHSSigma}).
An $E_{\infty}$ ring structure brings with it many extra tools and
much of the work of stable homotopy in the past two decades has
involved producing $E_{\infty}$ ring structures and $E_{\infty}$ ring
maps.

Recent work of Johnson and Noel \cite{JohnsonNoel}, however, shows that
maps out of $MU$ that come from $p$-typical orientations usually do
not commute with power operations.  As a consequence, many of the maps
of ring spectra out of $MU$ that are fundamental in the chromatic
picture of stable homotopy theory cannot be represented by
$E_{\infty}$ ring maps.  This mandates consideration of less rigid
structures than $E_{\infty}$ ring structures, and an obvious place to
start is the Boardman-Vogt hierarchy of $E_{n}$ structures, of which
$E_{\infty}$ is the apex.  An $E_{1}$ ring structure is also called an
$A_{\infty}$ ring structure (or associative $S$-algebra structure) and
retains all of the homotopy coherent associativity without the
commutativity.  An $E_{2}$ ring spectrum is homotopy commutative and
as $n$ gets higher, $E_{n}$ ring spectra become more coherently
homotopy commutative and have more of the power operations in an
$E_{\infty}$ ring spectrum.

The purpose of this paper is to study which genera of oriented
manifolds and stably almost complex manifolds are represented by maps
of $E_{n}$ ring spectra.  For reasons explained below, the easiest
case is when $n=2$, where we have the following results.

\begin{thm}\label{thmMSO}
Let $R$ be an even $E_{2}$ ring spectrum with
$1/2\in \pi_{0}R$.  Then every map of ring spectra $MSO\to R$ lifts to
a map of $E_{2}$ ring spectra $MSO\to R$.
\end{thm}

\begin{thm}\label{thmMU}
Let $R$ be an even $E_{2}$ ring spectrum.
Then every map of ring spectra $MU\to R$ lifts to a map of $E_{2}$
ring spectra $MU\to R$.
\end{thm}

Here ``even'' means that the homotopy groups are concentrated in even degrees, 
i.e., $\pi_{q}R=0$ for $q$ odd.  Examples of $E_{2}$ (or better) even
ring spectra include the Brown-Peterson spectrum $BP$, the Lubin-Tate
spectra $E_{n}$, and conjecturally, the truncated Brown-Peterson
spectra $BP\langle n\rangle$ and Johnson-Wilson spectra $E(n)$.  Each
of these spectra comes with a canonical map of ring spectra out of
$MU$ that is a $p$-typical orientation and that by the Johnson-Noel
result \cite[1.3,1.4]{JohnsonNoel} does not come from a map of
$E_{\infty}$ ring spectra (at small primes $p$, and conjecturally at
all primes).  Theorem~\ref{thmMU} shows that these maps do come from
maps of $E_{2}$ ring spectra.  The case of $BP$ seems particularly
worth highlighting as the coherence of the Quillen map $MU\to BP$ has
been an open question since the 1970's.

\begin{cor}\label{corQI}
The Quillen idempotent $MU_{(p)}\to MU_{(p)}$ and the Quillen map
$MU\to BP$ are represented by maps of $E_{2}$ ring spectra.
\end{cor}

In fact, since $BP$ may be constructed as the telescope over the
Quillen idempotent, this gives a new proof that $BP$ is an $E_{2}$
ring spectrum, independent of \cite{bmbp}; this argument appears in
detail in the first author's 2012 PhD thesis~\cite{thesis}.

Our techniques extend to give information for $E_{n}$ ring maps for
$n>2$ as well, especially the case $n=4$, but the picture is more
complicated.  For example, we prove the following result in
Section~\ref{secE4}.

\begin{thm}\label{thmE4}
There exists a map of ring spectra $MU\to MU$ that does not lift to
a map of $E_{4}$ ring spectra.
\end{thm}

To explain our approach to studying $E_{n}$ ring maps, it is easier if
we assume that the target ring spectrum $R$ is at least $E_{n+1}$.
(For definiteness we discuss the case of $MU$.) In this case,
we can take advantage of the fact that the Thom diagonal
\[
\tau \colon MU\to MU\sma BU_{+} = MU\sma \Sigma^{\infty}_{+}BU
\]
is an $E_{\infty}$ ring map \cite[p.~447]{lms} and that for any
$E_{n+1}$ ring spectrum $R$, the multiplication
\[
\mu \colon R\sma R\to R
\]
is an $E_{n}$ ring map \cite[1.6]{BFVTHH}.   Then for a
fixed $E_{n}$ ring map $\sigma
\colon MU\to R$ and a variable $E_{n}$ ring map $f\colon
\Sigma^{\infty}_{+}BU\to R$, the composite
\[
MU\overto {\tau} MU \sma \Sigma^{\infty}_{+}BU
\overto{\sigma \sma f}
R\sma R\overto{\mu}R
\]
is an $E_{n}$ ring map.  This induces a map from the space of $E_{n}$
ring maps $\Sigma^{\infty}_{+}BU\to R$ to the space of $E_{n}$ ring
maps $MU\to R$,
\[
\EnRing(\Sigma^{\infty}_{+}BU,R)\to \EnRing(MU,R).
\]
The usual algebraic argument (see Section~\ref{secThomIso}) then shows
that this map is an equivalence.

\begin{thm}\label{thmThomIso}
Let $R$ be an $E_{n+1}$ ring spectrum.  Then the space $\EnRing(MU,R)$
of
$E_{n}$ ring maps from $MU$ to $R$ is either empty or weakly
equivalent to the space $\EnRing(\Sigma^{\infty}_{+}BU,R)$ of $E_{n}$
ring maps from $\Sigma^{\infty}_{+}BU$ to $R$.
\end{thm}

The analogous theorem also holds for $MSO$ and $BSO$, and in general
for any Thom spectrum $M$ of an $E_{n}$ stable spherical (quasi)fibration
$X\to BF$ (see Theorem~\ref{thmgenthom}).  To be precise, the spaces
of $E_{n}$ ring maps $\EnRing(-,R)$ in the previous theorem are the
derived mapping spaces, i.e., the
homotopy types of the mapping spaces in the homotopy categories,
represented for example by the point set mapping space between a
cofibrant replacement (in the domain) and a fibrant replacement (in
the codomain) in a simplicial or topological model category of $E_{n}$
ring spectra.

We can identify the space of $E_{n}$ ring maps
$\Sigma^{\infty}_{+}BU\to R$ in more familiar terms.  Let $SL_{1}R$
denote the component of the zeroth space of $R$ corresponding to the
muliplicative identity element $1$ in $\pi_{0}R$.  Since $R$ is (in
particular) an
$E_{n}$ ring spectrum, the $E_{n}$ multiplication on $R$ induces an
$E_{n}$ structure on $SL_{1}R$.  The space of $E_{n}$ ring maps
$\Sigma^{\infty}_{+}BU\to R$ can be identified as the space of $E_{n}$
maps from $BU$ to $SL_{1}R$ \cite[IV.1.8]{mayeinf}, which is just the
space $\Top_{*}(B^{n}BU,B^{n}SL_{1}R)$ of based maps of topological
spaces $B^{n}BU\to B^{n}SL_{1}R$ (where $B^{n}$ denotes an $n$-fold
delooping functor).

In the case when $R$ is an $E_{\infty}$ ring spectrum, $SL_{1}R$ and
$B^{n}SL_{1}R$ are infinite loop spaces, and the Atiyah-Hirzebruch
spectral sequence
\begin{multline*}
E^{s,t}_{2} = H^{s}(B^{n}BU,\pi_{t}B^{n}SL_{1}R)
= H^{s}(B^{n}BU,\pi^{+}_{t-n}R)
\quad\Longrightarrow\\
\pi_{t-s}\Top_{*}(B^{n}BU,B^{n}SL_{1}R)=
\pi_{t-s}\EnRing(\Sigma^{\infty}_{+}BU,R)
\end{multline*}
calculates
the homotopy groups of $\EnRing(\Sigma^{\infty}_{+}BU,R)$.  Note that
$\pi_{t}B^{n}SL_{1}R=\pi_{t-n}SL_{1}R$ is $\pi_{t-n}R$ for $t-n>0$ and $0$ for
$t-n\leq 0$, and we use the notation $\pi^{+}_{t}R:=\pi_{t}SL_{1}R$ for these
groups.

When $R$ is just an $E_{n+1}$ ring spectrum as in
Theorem~\ref{thmThomIso},
$B^{n}SL_{1}R$ is a loop space, and the Postnikov tower of $B^{n}SL_{1}R$
is a sequence of principal fibrations of loop spaces of the form
\[
(B^{n}SL_{1}R)_{t}\to (B^{n}SL_{1}R)_{t-1}\to K(\pi^{+}_{t-n}R,t+1).
\]
Mapping $B^{n}BU$ into the tower $B^{n}SL_{1}R_{t}$ in the category
of based spaces,
we get a
tower of principal fibrations of loop spaces
\begin{multline*}
\Top_{*}(B^{n}BU,(B^{n}SL_{1}R)_{t})\to
\Top_{*}(B^{n}BU,(B^{n}SL_{1}R)_{t-1})\\
\to
\Top_{*}(B^{n}BU,K(\pi^{+}_{t-n}R,t+1))
\end{multline*}
with homotopy
limit weakly equivalent to $\Top_{*}(B^{n}BU,B^{n}SL_{1}R)$.  This
then again gives a spectral sequence for calculating the homotopy
groups of $\EnRing(\Sigma^{\infty}_{+}BU,R)$, whose $E_{2}$ term is
again
\[
E^{s,t}_{2}=H^{s}(B^{n}BU,\pi^{+}_{t-n}R)
\]
and which generalizes the Atiyah-Hirzebruch spectral sequence
displayed above.

In the case when $R$ is just an $E_{n}$ ring spectrum,
Theorem~\ref{thmThomIso} does not apply; nevertheless, we can identify
$\EnRing(MU,R)$ as the homotopy limit of a tower of principal
fibrations, using the Postnikov tower of $R$ in the category of
$E_{n}$ ring spectra (after replacing $R$ with its connective cover,
if necessary).  Basterra and the second author studied this tower in
\cite{bmthh,bmbp}, and using the work there, we prove the following
theorem in Section~\ref{secPT}.

\begin{thm}\label{thmPT}
Let $R$ be an $E_{n}$ ring spectrum, and if $n=1$, assume that
$\pi_{0}R$ is commutative.  Then the space of $E_{n}$ ring maps
from $MU$ to $R$ is weakly equivalent to the homotopy limit of
a tower of principal fibrations of the form
\[
\EnRing(MU,R_{q})\to \EnRing(MU,R_{q-1})\to \Top_{*}(B^{n}BU,K(\pi_{q}R,q+n+1))
\]
for $q\geq 1$.
\end{thm}

We can think of the previous theorem as giving an ``obstructed
spectral sequence'' (cf.~\cite{bouscosobs}) of the form
\[
E^{s,t}_{2}= H^{s}(B^{n}BU,\pi^{+}_{t-n}R)
\quad\Longrightarrow\quad
\pi_{t-s}\EnRing(MU,R)
\]
(for $t=q+n$).
In particular, it then gives an approach to calculating
$\pi_{0}\EnRing(MU,R)$, which we apply (in the generalized form of
Theorem~\ref{thmgenpt}) in Section~\ref{secMU} to prove
Theorems~\ref{thmMSO} and~\ref{thmMU}.

The discussion above shows why the cases of $E_{2}$ and $E_{4}$ maps
are the most tractable: For $n=2$ and $n=4$, $H^{*}(B^{n}BU)$ consists
of finitely generated free abelian groups and is concentrated in even
degrees.  In particular, when $\pi_{*}R$ is concentrated in even
degrees, the obstructions to lifting maps up the Postnikov tower
vanish and we can compute $\pi_{0}\EnRing (MU,R)$ as
\[
H^{n+2}(B^{n}BU,\pi_{2}R)\times
H^{n+4}(B^{n}BU,\pi_{4}R)\times \dotsb
\]
(as a set: $\pi_{0}\EnRing (MU,R)$ has no natural structure).
We cannot expect this to hold in general if $n\neq 2,4$.

\subsection*{Conventions}
Throughout this paper, the word ``space'' means compactly generated
weak Hausdorff space and $\Top$ denotes the category of such spaces.
We work in one of the modern categories of spectra, either symmetric
spectra (of spaces), orthogonal spectra, or EKMM $S$-modules and we
have written the details so that they work in any one of these
categories when no one is specified.  The word ``spectra'' means
(objects in) any one of these categories, and we write ``LMS spectra''
for (objects in) the category called spectra in \cite{lms}.

\subsection*{Acknowledgments}
Parts of this paper formed part of the first author's 2012 PhD thesis
at Indiana University \cite{thesis}.  He would like to warmly thank the senior author for his kind mentorship and ongoing guidance.  The authors would like to thank
Matt Ando, Maria Basterra, Andrew Blumberg, Mike Hopkins, Niles
Johnson, Nitu Kitchloo, and Ayelet Lindenstrauss for helpful conversations and remarks.

\section{The Homotopy Theory of $E_{n}$ Ring Spectra}
\label{secmodel}

We use this section to review the background on the homotopy
theory of $E_{n}$ ring spectra that we need for later sections.  Most
of the review consists of recording facts about model categories of
operadic algebras that are well-known to experts but scattered through
the literature and difficult to find in the precise form we need.  We
claim no originality for theorems in this section.

For much of the work in this paper, we take ``$E_{n}$ ring spectrum'' to
mean an algebra over the Boardman-Vogt little $n$-cubes operad
$\oC_{n}$ \cite[2.49]{BV}, \cite[\S4]{MayGILS} of spaces; however, in
part of Section~\ref{secThomIso}, we work instead with an $E_{n+1}$
ring spectrum that is an algebra over the tensor product operad
$\oC_{n}\otimes \Ass$.  The first background result we need is
therefore the well-known fact that we can model the homotopy category
of $E_{n}$ ring spectra using any $E_{n}$ operad, i.e., any
$\Sigma$-free operad (with paracompact Hausdorf underlying spaces)
that is weakly equivalent through operad maps to $\oC_{n}$.  The
following two theorems proved in Section~\ref{secpfmodel} establish
this fact.

\begin{thm}\label{thmOpModCat}
Let $\Spectra$ denote either the model category $\Sigma_{*}\aS$ of symmetric
spectra or the model category $\aI\aS$ of orthogonal spectra with their
positive stable model structures \cite[\S14]{MMSS} or the model category
$\EKMM$ of EKMM $S$-modules with its standard model structure
\cite[VII\S4]{ekmm}.  Let $\oO$ be an operad in spaces.  Then the
category $\Spectra[\oO]$ of $\oO$-algebras in $\Spectra$ is a
topological closed model category with fibrations and weak
equivalences created in $\Spectra$.
\end{thm}

\begin{thm}\label{thmChangeOp}
For $\Spectra$ as in Theorem~\ref{thmOpModCat}, and $\phi\colon \oO\to \oO'$ a
map of operads, the pushforward
(left Kan extension) and pullback functors
\[
L_{\phi}\colon \Spectra [\oO]\tofrom \Spectra[\oO']\noloc R_{\phi}
\]
form a Quillen adjunction, which is a Quillen equivalence if (and only
if) each $\phi(n)\colon \oO(n)\to \oO'(n)$ is a (non-equivariant)
stable equivalence.
\end{thm}

In the course of proving the previous theorems, we develop the tools
needed to deduce the following useful technical result. Note that the
initial $\oO$-algebra is $\oO(0)_{+}\sma S$.

\begin{thm}\label{thmOpCof}
Let $\Spectra$ and $\oO$ be as in Theorem~\ref{thmOpModCat}, and
assume that each $\oO(n)$ is a retract of a free $\Sigma_{n}$-cell
complex. If $A$ is a cofibrant $\oO$-algebra, then $\oO(0)_{+}\sma
S\to A$ is a cofibration in $\Spectra$.
\end{thm}

We need two more results geared towards using the Thom diagonal in the
context of $E_{n}$ ring spectra.  For a fibration of spaces $f\colon
B\to BF$ (where $BF$ denotes the classifying space for stable
spherical fibrations), Lewis constructed the Thom spectrum $Mf$ as an
LMS spectrum \cite[IX.3.2]{lms} and showed that when $f$ is a map of
$\oO$-spaces (for an operad $\oO$ with a map to the linear isometries
operad $\oL$), $Mf$ is naturally an $\oO$-algebra in the category of
LMS spectra \cite[IX.7.1]{lms}.  It follows that the Thom diagonal
$Mf\to Mf \sma B$ is a map of $\oO$-algebras.  Instead of re-proving
this in the context of a modern category of spectra, we just transport
this construction and this map using the following well-known
comparison theorems across the different categories of spectra.

\begin{thm}\label{thmOpComp1}
Let $\oO$ be an operad of spaces.  In the Quillen equivalences
\begin{gather*}
\bP\colon \Sigma_{*}\aS \tofrom \aI\aS\noloc \bU\\
\bN\colon \aI\aS \tofrom \EKMM\noloc \bN^{\#}
\end{gather*}
of \cite[p.~442]{MMSS} and \cite[I.1.1]{MMM}, all four functors
preserve $\oO$-algebras and induce Quillen equivalences
\begin{gather*}
\bP\colon \Sigma_{*}\aS[\oO] \tofrom \aI\aS[\oO]\noloc \bU\\
\bN\colon \aI\aS[\oO] \tofrom \EKMM[\oO]\noloc \bN^{\#}
\end{gather*}
on the categories of $\oO$-algebras.
\end{thm}

\begin{thm}\label{thmOpComp2}
Let $\oO$ be an operad of spaces.  Then the category $\LMS[\oO\times
\oL]$ of $(\oO\times \oL)$-algebras in LMS spectra is a topological
closed model category with fibrations and weak equivalences created in
LMS spectra. Moreover:
\begin{enumerate}
\item $\LMS[\oO\times \oL]$ is equivalent to the category
$\LMS[\bL][\oO]$ of $\oO$-algebras in EKMM $\bL$-spectra \cite[Ch.~I]{ekmm}.
\item The forgetful functor from EKMM $S$-modules to EKMM $\bL$-spectra and its
right adjoint $S\sma_{\oL}(-)$ both preserve $\oO$-algebras; the unit
and counit of this adjunction are both natural weak equivalences.
\item The right adjoint $F_{\oL}(S,-)\colon \EKMM\to \LMS[\bL]$ of
$S\sma_{\oL}(-)$ also preserves $\oO$-algebras; the unit and counit of
this adjunction are both natural weak equivalences.
\item The adjunction
\[
S\sma_{\oL}(-)\colon \LMS[\bL][\oO]\tofrom \EKMM[\oO]\noloc F_{\oL}(S,-)
\]
is a Quillen equivalence.
\end{enumerate}
\end{thm}

The proof of the model structure in Theorem~\ref{thmOpComp2} is given
in Section~\ref{secpfmodel} with the proof of the model structures in
Theorem~\ref{thmOpModCat}.  The proof of the remaining statements in
Theorems~\ref{thmOpComp1} and~\ref{thmOpComp2} are now easy from the
other theorems in the section.

\begin{proof}[Proof of Theorem~\ref{thmOpComp1}]
All four functors are lax symmetric monoidal and therefore preserve
operadic algebra structures.  Since fibrations on the algebra
categories are created in the underlying categories of spectra (i.e.,
in symmetric spectra, orthogonal spectra, or $S$-modules, as the case
may be), the
adjunctions on algebra categories are automatically Quillen
adjunctions.  To prove they are Quillen equivalences, by
\cite[A.2.(ii)]{MMSS}, it suffices to show that the derived functors
are equivalences of homotopy categories.  Applying
Theorem~\ref{thmChangeOp}, it suffices to consider the case when $\oO$
satisfies the hypothesis of Theorem~\ref{thmOpCof}, i.e., each
$\oO(n)$ is a $\Sigma_{n}$-cell complex.  In this case, every
cofibrant $\oO$-algebra is cofibrant in the underlying category of
spectra lying under $\oO(0)_{+}\sma S$.  Since the unit of the adjunction
is a weak equivalence for $\oO_{+}(0)\sma S$, the unit of the adjunction
is a weak equivalence for the codomain of any cofibration with domain
$\oO(0)_{+}\sma S$ (see for example the proof of \cite[10.3]{MMSS} and
the proof of \cite[I.3.5]{MMM}).  In particular the unit of the
adjunction is a weak equivalence for cofibrant $\oO$-algebras.
It then follows from \cite[A.2.(iii)]{MMSS} that the
Quillen adjunctions on algebra categories are Quillen equivalences.
\end{proof}

\begin{proof}[Proof of Theorem~\ref{thmOpComp2}.(i--iv)]
As in \cite[I\S4]{ekmm}, let $\bL$ denote the monad
$\oL(1)\ltimes(-)$ on the category of LMS spectra.  If we write $\rO$
for the free $\oO$-algebra functor on EKMM $\bL$-spectra, then
\begin{align*}
\rO\bL X
&=\bigvee_{n\geq 0}\oO(n)_{+}\sma_{\Sigma_{n}}(\bL X)^{\sma_{\oL}n}\\
&=\bigvee_{n\geq 0}\oO(n)_{+}\sma_{\Sigma_{n}}
  (\oL(n)\ltimes X^{\barwedge n})\\
&=\bigvee_{n\geq 0}((\oO(n)\times \oL(n))\ltimes X^{\barwedge n})/\Sigma_{n}
\end{align*}
is the free $(\oO\times \oL)$-algebra on $X$; an easy composite of
monads argument \cite[II.6.1]{ekmm}, then shows that the category of
$\oO$-algebras in $\LMS[\bL]$ is equivalent to the category of
$(\oO\times \oL)$-algebras in $\LMS$, which is~(i).  For (ii)
and~(iii), the adjunctions are \cite[II.1.3]{ekmm}. The fact that all
four functors are lax symmetric monoidal \cite[II.1.1]{ekmm}
shows that they preserve $\oO$-algebra structures, and the remaining
facts are \cite[I.8.5.(iii)]{ekmm} and \cite[I.8.7]{ekmm}.  For~(iv),
the adjunction is a Quillen adjunction because the adjunction on the
underlying categories $\EKMM$ and $\LMS[\bL]$ is a Quillen
adjunction \cite[VII.4.6]{ekmm}, and the adjunction is a Quillen
equivalence since both the unit and counit are weak equivalences on
all objects.
\end{proof}

\section{Proof of Theorem~\ref{thmThomIso}}
\label{secThomIso}

In this section, we prove Theorem~\ref{thmThomIso}, which relates the
space of $E_{n}$ ring maps out of $MU$ to the space of $E_{n}$ ring
maps out of $\Sigma^{\infty}_{+}BU$.  We view this theorem as the
$E_{n}$ ring version of the Thom isomorphism: The usual Thom
isomorphism relates maps of spectra out of $MU$ to maps of spectra out
of $\Sigma^{\infty}_{+}BU$.  Indeed, our proof of
Theorem~\ref{thmThomIso} generalizes to any $E_{n}$ ring Thom spectrum; see
Theorem~\ref{thmgenthom} below.  Since the theorem concerns derived
mapping spaces, the proof requires a certain amount of technical work
in the model category of $E_{n}$ ring spectra; however, for a
statement about maps in the homotopy category ($\pi_{0}$ of the derived
mapping space) a simpler argument in the homotopy category
suffices.  We give the homotopy category argument first as an
explanation and guide to the slightly more complicated model category
argument.

We work in the context of an $E_{n}$ ring Thom spectrum, defined as
follows.  Let $G$ denote either $O=\bigcup O(n)$, the infinite
orthogonal group, or $F=\bigcup F(n)$, the grouplike monoid of stable
self-homotopy equivalences of spheres.  (The monoid $F(n)$ is the
space of self-maps of $S^{n}$ that are homotopy equivalences and that
fix $0$ and $\infty$.)  Associated to any ``good'' map $f\colon X\to
BG$ is a Thom spectrum $M=Mf$ \cite[IX.3.2]{lms}; here ``good'' is the
technical condition of \cite[p.~423]{lms}: It is the empty condition
when $G=O$ and when $G=F$ it is satisfied in particular when $f$ is a
Hurewicz fibration, which we can always assume without loss of
generality \cite[pp.~411--412,443]{lms}.  The classifying space $BG$ is
an $E_{\infty}$ space for the linear isometries operad $\oL$
\cite{BV-HE}.  When $X$ is an $\oO\times\oL$-space for some
operad $\oO$ and $f$ is an $\oO\times \oL$-space map, the Thom
spectrum $M$ inherits the structure of an $\oO\times \oL$-spectrum.
In the particular case when $\oO$ is the little $n$-cubes operad
$\oC_{n}$, we call this an $E_{n}$ ring Thom spectrum.

\begin{defn}\label{defEnThom}
Let $\oC_{n}$ denote the Boardman-Vogt little $n$-cubes operad
\cite[2.49]{BV}, \cite[\S4]{MayGILS}.  An \term{$E_{n}$ ring Thom
spectrum} is the Thom spectrum of a $\oC_{n}\times \oL$-space map $X\to
BO$ or a ``good'' $\oC_{n}\times \oL$-space map $X\to BF$, viewed as an
$E_{n}$ ring spectrum.
\end{defn}

An $E_{n}$ ring Thom spectrum is then canonically a $\oC_{n}$-algebra
in EKMM $\bL$-spectra and (by applying $S\sma_{\oL}(-)$) canonically
weakly equivalent to a $\oC_{n}$-algebra in EKMM $S$-modules
(Theorem~\ref{thmOpComp2}), but up to weak equivalence, we can regard
it as a $\oC_{n}$-algebra in any of the modern categories of spectra
(Theorem~\ref{thmOpComp1}).  We fix one of the modern categories of
spectra, denoting it $\Spectra$ (calling its objects ``spectra''), and
write $\Spectra[\oC_{n}]$ for $\oC_{n}$-algebras in this category
(calling its objects ``$E_{n}$ ring spectra'').  As a topological
model category (Theorem~\ref{thmOpModCat}), the category of $E_{n}$
ring spectra has a nice theory of derived mapping spaces, constructed
for example as the mapping space out of a cofibrant replacement and
into a fibrant replacement. We use $\EnRing$ to denote the derived
mapping spaces; the homotopy
category of $E_{n}$ ring spectra is then $\pi_{0}\EnRing$.  The main
theorem of  this section is the following generalization of
Theorem~\ref{thmThomIso}.

\begin{thm}\label{thmgenthom}
Let $M$ be the $E_{n}$ ring Thom spectrum associated to an $E_{n}$
space map $f\colon X\to BG$ for some connected $E_{n}$ space $X$, and
let $R$ be an $E_{n+1}$ ring spectrum.  Then the derived space of maps
of $E_{n}$ ring spectra from $M$ to $R$ is either
empty or weakly equivalent to the derived space of maps of $E_{n}$
ring spectra from $\Sigma^{\infty}_{+}X$ to $R$,
\[
\EnRing(M,R)\simeq
\EnRing(\Sigma^{\infty}_{+}X,R).
\]
\end{thm}

At its core, the argument is a straightforward algebraic argument,
which gets somewhat obscured by technical details.  To outline and
explain the argument, we first prove the following easier
theorem.

\begin{thm}\label{thmgenthom0}
Let $M$ be the $E_{n}$ ring Thom spectrum associated to an $E_{n}$
space map $f\colon X\to BG$ for some connected $E_{n}$ space $X$, and
let $R$ be an $E_{n+1}$ ring spectrum.  If there exists a map $M\to R$
in the homotopy category of $E_{n}$ ring spectra, then the set of maps
$M\to R$ in the homotopy category of $E_{n}$ ring spectra is in
one-to-one correspondence with the set of maps $\Sigma^{\infty}_{+}X$ to $R$,
in the homotopy category of $E_{n}$ ring spectra,
\[
\pi_{0}\EnRing(M,R)\iso \pi_{0}\EnRing(\Sigma^{\infty}_{+}X,R).
\]
\end{thm}

The proof of Theorem~\ref{thmgenthom0} is little more than an
application of the Thom isomorphism theorem and an exercise with
monoids and modules inside the homotopy category of $E_{n}$ ring spectra.
We note that if $A$ and $B$ are $E_{n}$ ring spectra, then $A\sma B$
(point-set smash product in $\Spectra$) is canonically an $E_{n}$ ring
spectrum with action of $\oC_{n}$ induced by using the diagonal map
$\oC_{n}\to \oC_{n}\times \oC_{n}$ and the actions on $A$ and $B$.
Since $\oC_{n}(0)=*$ and each space $\oC_{n}(m)$ is a free
$\Sigma_{m}$-cell complex,
cofibrant $E_{n}$ ring spectra are cofibrant objects in spectra under
$S$ (Theorem~\ref{thmOpCof}).
In particular, the smash product with a cofibrant $E_{n}$ ring
spectrum preserves weak equivalences in $\Spectra$, and it follows
that $\sma$ descends to a symmetric monoidal product on the homotopy
category of $E_{n}$ ring spectra, compatibly with the smash product in
the stable category.

Let $R$ be an $E_{n+1}$ ring spectrum (a $\oC_{n+1}$-algebra in
$\Spectra$).  We use the map of operads $\ell\colon \oC_{n}\to \oC_{n+1}$ that
sends a little $n$-cube $a$ to the little $n+1$-cube $a\times [0,1]$
to regard $R$ as an $E_{n}$ ring spectrum.  We also have a map of
operads $r \colon C_{1}\to \oC_{n+1}$ sending a little $1$-cube $b$
to the little $n+1$-cube $[0,1]^{n}\times b$; using $r$, for any
element $c$ of $\oC_{1}(m)$, we then get a map
\[
r(c)\colon R^{(m)}=R\sma \dotsb \sma R\to R,
\]
(where $R^{(m)}$ denotes the $m$-th smash power of $R$).  Because the
actions induced by $\ell$ and $r$ on $R$ satisfy the interchange law
\cite[1-1]{BFVTHH}, the map $r(c)$ is a map of $E_{n}$ ring spectra.
In particular, working in the homotopy category of $E_{n}$ ring
spectra and taking $c$ to be the element $\mu$ in $\oC_{1}(2)$
representing the standard multiplication, we see that $R$ is a monoid
for the smash product in the homotopy category of $E_{n}$ ring
spectra.  (We will henceforth omit the $r$ and write
$\mu\colon R\sma R\to R$ for this map.) We use the following
terminology for modules.

\begin{defn}\label{defmodule}
Let $R$ be an $E_{n+1}$ ring spectrum.  A \term{homotopical $R$-module
in $E_{n}$ ring spectra} is a left module for $R$ in the homotopy
category of $E_{n}$ ring spectra:  It consists of an $E_{n}$ ring
spectrum $N$ together with an action map
\[
\xi\colon R\sma N\to N
\]
in the homotopy category of $E_{n}$ ring spectra such that the
composite map
\[
S\sma N\to R\sma N\to N
\]
is the canonical isomorphism and the associativity diagram
\[
\xymatrix{%
R\sma R\sma N\ar[r]^-{\id_{R}\sma \xi}\ar[d]_-{\mu\sma \id_{N}}
&R\sma N\ar[d]^-{\xi}\\
R\sma N\ar[r]_-{\xi}&N
}
\]
commutes, where $\mu \colon R\sma R\to R$ is the multiplication
discussed above.  A map of homotopical $R$-modules
in $E_{n}$ ring spectra is a map in the homotopy category of $E_{n}$
ring spectra $N\to N'$ that commutes with the action maps.  We use the symbol $\Modho$
to denote the category of homotopical $R$-modules in $E_{n}$ ring spectra.
\end{defn}

We omit ``in $E_{n}$ ring spectra'' from the terminology for homotopical
$R$-modules when it is clear from context.  We have the usual
free/forgetful adjunction for these modules.

\begin{prop}\label{propfree}
Let $R$ be an $E_{n+1}$ ring spectrum.
The functor $R\sma(-)$ from the homotopy category of $E_{n}$ ring
spectra to the category of homotopical $R$-modules in $E_{n}$ ring
spectra is left adjoint to the forgetful functor: Maps in the homotopy
category of $E_{n}$ ring spectra from an $E_{n}$ ring spectrum $E$ to
a homotopical $R$-module $N$ are in one-to-one correspondence with
maps of homotopical $R$-modules from $R\sma E$ to $N$,
\[
\pi_{0}\EnRing(E,N)\iso \Modho(R\sma E,N).
\]
\end{prop}

\begin{proof}
The correspondence is the usual one: Given a map $h\colon E\to N$ in
the homotopy category of $E_{n}$ ring spectra, the composite
\[
R\sma E \overto{\id_{R}\sma h} R\sma N\overto{\xi} N
\]
is a map of homotopical $R$-modules
and given a map $k\colon R\sma E\to N$ of homotopical $R$-modules, the
map $E\to N$ is the composite map in the homotopy category of $E_{n}$
ring spectra
\[
E\iso S\sma E\to R\sma E\overto{k}N.
\]
An easy check shows these are inverse correspondences.
\end{proof}

When $M$ is the $E_{n}$ ring Thom spectrum associated to an $E_{n}$
space map $f\colon X\to BG$,  it follows from \cite[IX.7.1]{lms} (see
in particular the top of page~447 in \cite{lms}) that the Thom
diagonal
\[
\tau \colon M\to M\sma X_{+} = M\sma \Sigma^{\infty}_{+}X
\]
lifts to a natural map in the homotopy category of $E_{n}$ ring
spectra.  The following is then the $E_{n}$ ring spectrum version of the
homology Thom isomorphism.  Recall that a map $\sigma$ from a Thom
spectrum $M$
to a ring spectrum $R$ is an \term{orientation} when for every
point $x$ in $X$, the map $S\to R$ obtained by restricting $\sigma$ to
the Thom spectrum of $\{x\}$ represents a unit in the ring $\pi_{0}R$.  When
$X$ is connected, a map of ring spectra $M\to R$ is always an
orientation since the restriction $S\to R$ represents the identity
element in $\pi_{0}R$.

\begin{prop}\label{prophomthom}
Let $M$ be the $E_{n}$ ring Thom spectrum associated to an $E_{n}$
space map $f\colon X\to BG$ and let $R$ be an $E_{n+1}$ ring
spectrum.  If $\sigma \colon M\to R$ is a map in the homotopy category
of $E_{n}$ ring spectra and also an orientation, then the map
\[
M\overto{\tau}
M \sma \Sigma^{\infty}_{+}X
\overto{\sigma \sma \id_{\Sigma^{\infty}_{+}X}}
R\sma \Sigma^{\infty}_{+}X
\]
induces an isomorphism of homotopical $R$-modules in $E_{n}$-ring
spectra
\[
R\sma M\to R\sma \Sigma^{\infty}_{+}X.
\]
\end{prop}

\begin{proof}
As the composite map $M\to R\sma \Sigma^{\infty}_{+}X$ is a map in the
homotopy category of $E_{n}$ ring spectra, we get an induced map of
homotopical $R$-modules as displayed above by the free/forgetful
adjunction (Proposition~\ref{propfree}).  The question of it being
an isomorphism is a question in the stable category (after forgetting
the $E_{n}$ ring structures and just remembering the homotopical ring
spectrum structure on $R$), and this is just the usual homology
version of the Thom isomorphism, the map $R\sma M\to R\sma
\Sigma^{\infty}_{+}X$, being the geometric cap product with the
orientation $\sigma$.
\end{proof}

Theorem~\ref{thmgenthom0} is now an easy consequence.  Propositions~\ref{propfree}
and~\ref{prophomthom} give us bijections
\begin{multline*}
\pi_{0}\EnRing(M,R)\iso \Modho(R\sma M,R)
\iso \\
\Modho(R\sma \Sigma^{\infty}_{+}X,R)
\iso \EnRing(\Sigma^{\infty}_{+}X,R)
\end{multline*}
under the hypothesis that a map $\sigma \colon M\to R$ exists in the
homotopy category of $E_{n}$ ring spectra.  This completes the proof
of Theorem~\ref{thmgenthom0}.

We can prove Theorem~\ref{thmgenthom} by the same outline, but using
stricter algebraic structures.  Whereas an $E_{n+1}$ ring spectrum is
a monoid for the smash product in the homotopy category of $E_{n}$
ring spectra, it is only an $A_{\infty}$ monoid for the point-set
smash product of $E_{n}$ ring spectra. A monoid for the point-set
smash product of $E_{n}$ ring spectra is precisely an algebra over the
operad $\oC_{n}\otimes \Ass$ \cite[\S1.6]{BFVTHH}, where $\Ass$ is the
operad defining associative monoids.  Theorem~C of
\cite{BFVTHH} shows that $\oC_{n}\otimes \Ass$ is an $E_{n+1}$ operad,
and so given an $E_{n+1}$ ring spectrum $R$, we can find an equivalent
$\oC_{n}\otimes \Ass$-algebra $R'$, which we can regard as a monoid
for the point-set smash product of $E_{n}$ ring spectra.  We then have
the following point-set category of point-set modules.

\begin{defn}
Let $R$ be a monoid for the point-set smash product of $E_{n}$ ring
spectra, or equivalently, an algebra over the operad
$\oC'_{n+1}:=\oC_{n}\otimes \Ass$.  An \term{$R$-module in the category of
$E_{n}$ ring spectra} consists of an $E_{n}$ ring spectrum $N$ together
with an action map
\[
\xi\colon R\sma N\to N
\]
in the point-set category of $E_{n}$-ring spectra $\Spectra[\oC_{n}]$
such that the composite map
\[
S\sma N\to R\sma N\to N
\]
is the canonical isomorphism and the associativity diagram
\[
\xymatrix{%
R\sma R\sma N\ar[r]^-{\id_{R}\sma \xi}\ar[d]_-{\mu\sma \id_{N}}
&R\sma N\ar[d]^-{\xi}\\
R\sma N\ar[r]_-{\xi}&N
}
\]
commutes (in $\Spectra[\oC_{n}]$).  A map of $R$-modules
in $E_{n}$ ring spectra is a map $N\to N'$ in $\Spectra[\oC_{n}]$
that commutes with the action maps.  We denote the category of
$R$-modules in $E_{n}$ ring spectra as $\Modps$.
\end{defn}

The following theorem does the technical work in extending the outline
above for the proof of Theorem~\ref{thmgenthom}.

\begin{thm}\label{thmRModMod}
Let $R$ be a $\oC'_{n+1}$-algebra.  Then the category $\Modps$ of $R$-modules
in $E_{n}$ ring spectra is a topological closed model category with
the fibrations and weak equivalences created in $\Spectra[\oC_{n}]$.
\end{thm}

\begin{proof}
The topological model structure is a consequence of Theorem~\ref{thmgenmodel} below, which
generalizes Theorem~\ref{thmOpModCat} to operads in $\Spectra$.  Starting from $\Spectra$, the free functor from $\Spectra$ to $\Modps$
is
\[
R\sma \rCn X = R \sma \bigl(
\bigvee_{m\geq 0} \oC_{n}(m)_{+}\sma_{\Sigma_{m}}X^{(m)}
\bigr).
\]
This is the monad associated to the operad $\oR$ in $\Spectra$ defined
by $\oR(m)=R\sma \oC_{n}(m)_{+}$, with identity
\[
S\iso S\sma \{*\}_{+}\to R\sma \oC_{n}(1)_{+},
\]
equivariance from the equivariance of $\oC_{n}(m)$, and multiplication
\[
R\sma \oC_{n}(m)_{+}\sma
((R\sma \oC_{n}(j_{1})_{+})\sma \dotsb \sma
(R\sma \oC_{n}(j_{m})_{+}))
\to R\sma \oC_{n}(j)
\]
induced by the operadic multiplication on $\oC_{n}$, the
$\oC_{n}$-action on $R$,
\[
\oC_{n}(m)_{+}\sma_{\Sigma_{m}} R^{(m)}\to R
\]
and the multiplication $\mu \colon R\sma R\to R$ from the monoid
structure on $R$.  It follows that $\Modps$ is isomorphic to the
category of $\oR$-algebras, hence admits the topological model structure by Theorem~\ref{thmgenmodel}.
\end{proof}

\begin{cor}\label{corFree}
The free functor $R\sma(-)\colon \Spectra[\oC_{n}]\to \Modps$ and
forgetful functor $\Modps\to\Spectra[\oC_{n}]$ form a Quillen
adjunction.
\end{cor}

As we have already noted, the smash product with a cofibrant $E_{n}$
ring spectrum preserves all weak equivalences in $\Spectra$; it
follows that the derived functor of the free functor $R\sma (-)$ is
the derived smash product with $R$ after forgetting down to the
homotopy category of $E_{n}$ ring spectra or all the way down to the
stable category.   Combining the previous corollary with
Proposition~\ref{prophomthom}, we then get the following corollary.

\begin{cor}\label{corHomThom}
Let $M$ be the $E_{n}$ ring Thom spectrum associated to an $E_{n}$
space map $f\colon X\to BG$ and let $R$ be a $\oC'_{n+1}$-algebra.  If
$\sigma \colon M\to R$ is a map in the homotopy category
of $E_{n}$ ring spectra and also an orientation, then the map
\[
M\overto{\tau}
M \sma \Sigma^{\infty}_{+}X
\overto{\sigma \sma \id_{\Sigma^{\infty}_{+}X}}
R\sma \Sigma^{\infty}_{+}X
\]
in the homotopy category of $E_{n}$ ring spectra induces an
isomorphism in the homotopy category of $R$-modules in $E_{n}$-ring
spectra
\[
R\sma M\to R\sma \Sigma^{\infty}_{+}X.
\]
\end{cor}

Corollary~\ref{corHomThom} is what we need to prove Theorem~\ref{thmgenthom}.

\begin{proof}[Proof of Theorem~\ref{thmgenthom}]%
Let $R$ be an $E_{n+1}$ ring spectrum; we can then find an equivalent
$\oC'_{n+1}$-algebra $R'$ (which is in particular weakly equivalent as
an $E_{n}$ ring spectrum).  Without loss of generality, we can assume
that $R'$ is fibrant as an $\oC'_{n+1}$-algebra and therefore also as an
$E_{n}$ ring spectrum.  We choose cofibrant approximations $M'\to
M$ and $A\to \Sigma^{\infty}_{+}X$.  Suppose there exists a map $\sigma
\colon M\to R\simeq R'$ in the homotopy of $E_{n}$ ring spectra; then
since $X$ is connected, $\sigma$ is an orientation and
Corollaries~\ref{corFree} and~\ref{corHomThom} give us weak
equivalences of mapping
spaces
\[
\Spectra[\oC_{n}](M',R')\iso \Modps(R'\sma M',R')\simeq
\Modps(R'\sma A,R')\iso \Spectra[\oC_{n}](A,R').
\]
The composite is then a weak equivalence
\[
\EnRing (M,R)\simeq \EnRing(\Sigma^{\infty}_{+}X,R).\qedhere
\]
\end{proof}

\section{Proof of Theorem~\ref{thmPT}}
\label{secPT}

In this section we prove the following generalization of
Theorem~\ref{thmPT} from the introduction.  (See
Definition~\ref{defEnThom} for the definition of an $E_{n}$ ring Thom
spectrum.)

\begin{thm}\label{thmgenpt}
Let $M$ be an $E_{n}$ ring Thom spectrum associated to to an $E_{n}$
space map $f\colon X\to BG$ and assume that $X$ is connected.
Let $R$ be an $E_{n}$ ring spectrum, and if $n=1$, assume that
$\pi_{0}R$ is commutative.  Then the space $\EnRing (M,R)$ of $E_{n}$ ring maps
from $M$ to $R$ is weakly equivalent to the homotopy limit of
a tower of principal fibrations of the form
\[
\EnRing(M,R_{q})\to \EnRing(M,R_{q-1})\to \Top_{*}(B^{n}X,K(\pi_{q}R,q+n+1))
\]
for $q\geq 1$.
\end{thm}

We fix $X$, $M$, and $R$ as in the theorem, and we assume without loss
of generality that $R$ is fibrant.  Choose a cofibrant approximation
$M'\to M$ in the category of $E_{n}$ ring spectra.  Let $c\colon \bar
R\to R$ be a connective cover, i.e., $\bar R$ is connective
($\pi_{q}\bar R=0$ for $q<0$) and $c$ induces an isomorphism on
non-negative homotopy groups.  (The connective cover can be
constructed by applying the small objects argument as if to construct
a cofibrant approximation but only using the non-negative dimensional
cells; alternatively it can be constructed using multiplicative
infinite loop space theory applied to the zeroth space of $R$
\cite[\S4]{May-mult}).  We assume without loss of generality that
$\bar R$ is fibrant and also cofibrant in the category
$\Spectra[\oC_{n}]$ of $E_{n}$ ring spectra.  Then the derived mapping
spaces $\EnRing(M,R)$ and $\EnRing(M,\bar R)$ may be constructed as
the point set mapping spaces $\Spectra[\oC_{n}](M',R)$ and
$\Spectra[\oC_{n}](M',\bar R)$, respectively.  The following
observation reduces to the connective case.

\begin{prop}
The map $c\colon \bar R\to R$ induces a weak equivalence
$\EnRing(M,\bar R)\to \EnRing(M,R)$.
\end{prop}

\begin{proof}
This can be deduced from the results in Section~4 of \cite{May-mult}.
A more modern approach is to observe that the cofibrant approximation
$M'\to M$ can be built starting
from $S$ entirely using ``positive dimensional cells'', i.e., cells of
the form
\[
\rCn S^{q}_{c}\to \rCn CS^{q}
\qquad\text{or}\qquad
\rCn F_{m}S^{m+q}_{+}\to \rCn F_{m}D^{m+q+1}_{+}
\]
(the former when $\Spectra$ is EKMM $S$-modules, the latter when
$\Spectra$ is symmetric or orthogonal spectra) for $q\geq 0$, where
$\rCn$ denotes the free $\oC_{n}$ algebra functor.
\end{proof}

Let $H=H\pi_{0}R$ be a fibrant model of the Eilenberg-Mac Lane ring
spectrum; the hypothesis of Theorem~\ref{thmgenpt} allows us to choose
the model $H$ with the structure of an $E_{\infty}$ ring spectrum (or
commutative $S$-algebra).  The following is an easy induction argument
on the cell structure of a cofibrant approximation of the domain.

\begin{prop}\label{prophtydisc}
For any connective $E_{n}$ ring spectrum $E$, the mapping space
$\EnRing(E,H)$ is homotopy discrete with $\pi_{0}$ the set of ring
maps from $\pi_{0}E$ to $\pi_{0}H$.
\end{prop}

The hypothesis that $X$ is connected implies that $\pi_{0}M$ is either
$\bZ$ or $\bZ/2$, and so it follows that $\EnRing(M,H)$ is either
empty or weakly contractible.  In the case when $\EnRing(M,H)$ is
empty, so is $\EnRing(M,R)$ and Theorem~\ref{thmgenpt} holds for
trivial reasons.  We henceforth restrict to the case when
$\EnRing(M,H)$ is weakly contractible and fix a map $M'\to H$.
Likewise we fix a map $\bar R\to H$ representing the identity map on
$\pi_{0}\bar R=\pi_{0}H$.  Writing $\EnRing_{/H}$ for the derived
mapping space in the category of $E_{n}$ ring spectra lying over $H$,
we then have the following result.

\begin{prop}
The forgetful map
\[
\EnRing_{/H}(M',\bar R)\to \EnRing(M',\bar R)\simeq
\EnRing(M,R)
\]
is a weak equivalence.
\end{prop}

Thus, to study $\EnRing(M,R)$, we can study the space of maps in the
category $\Spectra[\oC_{n}]/H$ of $E_{n}$ ring spectra lying over the
$E_{\infty}$ ring spectrum $H$.  This is precisely the situation
studied in \cite[\S4]{bmbp}.  In particular, Theorem~4.2 of \cite{bmbp}
constructs a Postnikov tower for $\bar R$ as a tower of principal
fibrations in $\Spectra[\oC_{n}]/H$.  Specifically, we start with
$\bar R\to R_{0}\to H$ a cofibration followed by an acyclic fibration.
Then for each $q>0$, we can inductively construct $R_{q}$ as (a
cofibrant approximation of) the
homotopy pullback of maps
\[
\xymatrix{%
&H\ar[d]\\
R_{q-1}\ar[r]_-{k^{n}_{q}}&(H\vee \Sigma^{q+1} H\pi_{q}R)_{f},
}
\]
where $(H\vee \Sigma^{q+1} H\pi_{q}R)_{f}$ denotes a fibrant
approximation of the ``square zero'' $E_{\infty}$ ring spectrum $H\vee
\Sigma^{q+1} H\pi_{q}R$ (meaning that the multiplication on the summand
\[
\Sigma^{q+1} H\pi_{q}R \sma \Sigma^{q+1} H\pi_{q}R
\]
is the trivial map).  Using the path space construction of the
homotopy pullback, we can arrange that the map $R_{q}\to R_{q-1}$ is a
fibration.   The map $k^{n}_{q}$ is chosen so that there is
an induced map $\bar
R\to R_{q}$ that is an isomorphism on homotopy groups in dimension $q$
and below; for formal reasons, the underlying map of spectra
\[
R_{q-1}\to H\vee \Sigma^{q+1} H\pi_{q}R \to \Sigma^{q+1} H\pi_{q}R
\]
is then the $q$-th $k$-invariant $k_{q}$ in the Postnikov tower for
$\bar R$.  Looking at the space of maps in $\Spectra[\oC_{n}]/H$ from
$M'$ into these squares and this tower, we get the following result.

\begin{thm}
The space $\EnRing(M,R)$ of $E_{n}$ ring maps
from $M$ to $R$ is weakly equivalent to the homotopy limit of
a tower of principal fibrations of the form
\[
\EnRing(M,R_{q})\to \EnRing(M,R_{q-1})\to
\EnRing(M,H\vee \Sigma^{q+1}H\pi_{q}R)
\]
for $q\geq 1$.
\end{thm}

Since $H\vee H\pi_{q}R$ is an $E_{\infty}$ ring spectrum, and we have
a canonical map in the homotopy category of $E_{n}$ ring spectra $M\to
H\to H\vee H\pi_{q}R$, we can apply Theorem~\ref{thmgenthom} to obtain
a weak equivalence
\[
\EnRing(M,H\vee \Sigma^{q+1}H\pi_{q}R) \simeq
\EnRing(\Sigma^{\infty}_{+}X,H\vee \Sigma^{q+1}H\pi_{q}R).
\]
Using primarily Theorem~1.3 of \cite{bmthh}, we prove the following
theorem, which then completes the proof of Theorem~\ref{thmgenpt}.

\begin{thm}\label{thmbmmain}
$\EnRing(\Sigma^{\infty}_{+}X,H\vee \Sigma^{q+1}H\pi_{q}R)
\simeq \Top_{*}(B^{n}X,K(\pi_{q}R,q+n+1))$.
\end{thm}

\begin{proof}
We note that $\Sigma^{\infty}_{+}X$ comes with a canonical map to $S$
induced by the map
$X\to *$.  From Proposition~\ref{prophtydisc}, we see that
$\EnRing(\Sigma^{\infty}_{+}X,H)$ is weakly contractible and hence
that the map
\[
\EnRing_{/H}(\Sigma^{\infty}_{+}X,H\vee \Sigma^{q+1}H\pi_{q}R)\to
\EnRing(\Sigma^{\infty}_{+}X,H\vee \Sigma^{q+1}H\pi_{q}R)
\]
is a weak equivalence.  Pulling back along the map $S\to H$ is the
right adjoint in a Quillen adjunction between the category of $E_{n}$
ring spectra lying over $S$ and the category of $E_{n}$ ring spectra
lying over $H$, and so we get a weak equivalence
\[
\EnRing_{/S}(\Sigma^{\infty}_{+}X,S\vee \Sigma^{q+1}H\pi_{q}R)\to
\EnRing_{/H}(\Sigma^{\infty}_{+}X,H\vee \Sigma^{q+1}H\pi_{q}R)
\]
where $S\vee \Sigma^{q+1}H\pi_{q}R$ has the square zero multiplication.
In the notation of \cite[\S7]{bmthh},
\[
S\vee \Sigma^{q+1}H\pi_{q}R=KZ (\Sigma^{q+1}H\pi_{q}R).
\]
where $KZ$ is the square zero multiplication functor from spectra to
$E_{n}$ ring spectra lying over $S$.
The Quillen adjunctions of \cite[7.1,7.2]{bmthh}, then give us a weak
equivalence
\begin{multline*}
\EnRing_{/S}(\Sigma^{\infty}_{+}X,S\vee \Sigma^{q+1}H\pi_{q}R)
\simeq \Stable(\Sigma^{-n}I^{\rd}(B^{n}\Sigma^{\infty}_{+}X),
\Sigma^{q+1}H\pi_{q}R)\\
\simeq
\Stable(I^{\rd}(B^{n}\Sigma^{\infty}_{+}X),\Sigma^{q+n+1}H\pi_{q}R)
\end{multline*}
where $\Stable$ denotes the derived space of maps of spectra,
$I^{\rd}$ is the homotopy fiber of the augmentation, and $B^{n}$ is
the iterated bar construction for $E_{n}$ ring spectra lying over $S$
constructed in
\cite{bmthh}.  This bar construction commutes with the unbased
suspension spectrum functor, so we get a weak equivalence
\[
\Stable(I^{\rd}(B^{n}\Sigma^{\infty}_{+}X),\Sigma^{q+n+1}H\pi_{q}R)
\simeq
\Stable(I^{\rd}\Sigma^{\infty}_{+}B^{n}X,\Sigma^{q+n+1}H\pi_{q}R).
\]
Since the augmentation $\Sigma^{\infty}_{+}B^{n}X\to S$ is split by
the unit $S\to \Sigma^{\infty}_{+}B^{n}X$, we can identify the
homotopy fiber of the augmentation as the cofiber of the unit.  This
gives us a weak equivalence
\[
\Stable(I^{\rd}\Sigma^{\infty}_{+}B^{n}X,\Sigma^{q+n+1}H\pi_{q}R)
\simeq
\Stable(\Sigma^{\infty}B^{n}X,\Sigma^{q+n+1}H\pi_{q}R)
\]
where $B^{n}X$ has its usual basepoint.  The usual suspension spectrum,
zeroth space (i.e., underlying infinite loop space) adjunction then
gives the weak equivalence
\begin{multline*}
\Stable(\Sigma^{\infty}B^{n}X,\Sigma^{q+n+1}H\pi_{q}R)
\simeq
\Top_{*}(B^{n}X,\Omega^{\infty}(\Sigma^{q+n+1}H\pi_{q}R))\\
\simeq
\Top_{*}(B^{n}X,K(\pi_{q}R,q+n+1))
\end{multline*}
completing the proof.
\end{proof}

\section{Proof of Theorems~\ref{thmMSO} and~\ref{thmMU}}
\label{secMU}

The entirety of this section is devoted to the proof of
Theorems~\ref{thmMSO} and~\ref{thmMU}.  We fix an even $E_{2}$ ring
spectrum $R$ and carry
over the notation $\bar R$, $R_{q}$, and $H$ from the last section:
$\bar R\to R$ is a connective cover, and
\[
\bar R\to \dotsb \to R_{q}\to R_{q-1}\to \dotsb \to R_{0}\simeq H
\]
is a Postnikov tower in the category of $E_{2}$ ring spectra.
In the case of $MSO$ we assume that $\pi_{0}R$ contains $1/2$.

Our proof is an inductive argument up the Postnikov tower.  Both
arguments are essentially the same, so we do the case of $MU$ in
detail, with the changes necessary for $MSO$ in Remark~\ref{remMSO}
below. We write
$\HoRing(MU,R_{q})$ for the set of maps of ring spectra (in the stable
category) from $MU$ to $R_{q}$.  The inductive hypothesis (indexed on
integers $q\geq 0$) is the following:
\begin{enumerate}
\item The forgetful map $\pi_{0}\EnRing[2](MU,R_{q})\to
\HoRing(MU,R_{q})$ is surjective.
\item For $q>0$, the map from $\pi_{0}\EnRing[2](MU,R_{q})$ to
the fiber product of the maps
\[
\pi_{0}\EnRing[2](MU,R_{q-1})\to \HoRing(MU,R_{q-1})\from
\HoRing(MU,R_{q})
\]
is surjective.
\item $\pi_{1}(\EnRing[2](MU,R_{q}),f)$ is trivial for all
basepoints $f$.
\end{enumerate}
Under the hypothesis that $R$ is even, we have
\[
\HoRing(MU,R)\iso \lim \HoRing(MU,R_{q}).
\]
Inductive hypothesis~(iii) implies
\[
\pi_{0}\EnRing[2](MU,R)\iso \lim \pi_{0}\EnRing[2](MU,R_{q})
\]
and inductive hypotheses~(i) and~(ii) then imply that the map
$\pi_{0}\EnRing[2](MU,R)\to \HoRing(MU,R)$ is surjective, which will
complete the proof of Theorem~\ref{thmMU}.

In the base case $q=0$, $R_{0}\simeq H$ and both
$\pi_{0}\EnRing[2](MU,H)$ and $\HoRing(MU,H)$ consist of a single
point.  Thus, inductive hypothesis~(i) holds.  Inductive
hypothesis~(ii) is empty in this case, and inductive hypothesis~(iii)
holds since $\EnRing[2](MU,H)$ is weakly contractible.

For $q\geq 1$, it suffices to consider the case when $q$ is even since
the map $R_{q}\to R_{q-1}$ is a weak equivalence when $q$ is odd.  We
look at the fiber sequence
\begin{equation}\label{eqfibseq}
\to
\EnRing[2](MU,R_{q})\to \EnRing[2](MU,R_{q-1})\to
\EnRing[2](MU,H\vee \Sigma^{q+1}H\pi_{q}R)
\end{equation}
and use the identification of Theorems~\ref{thmgenpt}
and~\ref{thmbmmain} of $\EnRing[2](MU,H\vee \Sigma^{q+1}H\pi_{q}R)$
with
\[
\Top_{*}(B^{2}BU,K(\pi_{q}R,q+3))
\simeq
\Top_{*}(BSU,K(\pi_{q}R,q+3)).
\]
This then gives us a computation of the homotopy groups of the base space:
\begin{multline}\label{eqpie2}
\pi_{m}\EnRing[2](MU,H\vee \Sigma^{q+1}H\pi_{q}R)\iso\\
\pi_{m}\Top_{*}(BSU,K(\pi_{q}R,q+3))=\tilde H^{q+3-m}(BSU;\pi_{q}R).
\end{multline}
The integral cohomology of $BSU$ is well-known: It is a polynomial
algebra on the Chern classes $c_{2}$, $c_{3}$, etc.  We see that the
base space of the fibration~\eqref{eqfibseq} is therefore
connected with non-zero homotopy groups only in odd degrees.  The
inductive hypothesis~(iii) for $q-1$ that
$\pi_{1}(\EnRing[2](MU,R_{q-1}),f)$ is trivial for all basepoints $f$ now
implies the inductive hypothesis~(iii) for $q$ that
$\pi_{1}(\EnRing[2](MU,R_{q}),g)$ is trivial for all basepoints $g$.

For the inductive steps~(i) and~(ii), we need to relate our fiber
sequence~\eqref{eqfibseq} and the map $\HoRing(MU,R_{q})\to
\HoRing(MU,R_{q-1})$.  For this, we use the well-known fact that
a ring spectrum map $f\colon MU\to R_{q}$ is completely determined by
its restriction to $MU(1)$ as a map in the stable category
(q.v.~\cite[II.4.6]{AdamsBlueBook}, \cite[10.10]{Dold}), where
$MU(1)\simeq \Sigma^{-2}\bC P^{\infty}$ is the Thom spectrum of
$BU(1)$.  Writing $\Stable(MU(1),R_{q})$ for the derived space of maps
in $\Spectra$ from $MU(1)$ to $R_{q}$, let $\Stable(MU(1),R_{q})_{u}$
denote the subspace of components that map to the component of the unit
map $S\to R$ in $\Stable(S,R_{q})$  (via the inclusion of $S$
in $MU(1)$).  Then the map
\[
\HoRing(MU,R_{q})\to \pi_{0}\Stable(MU(1),R_{q})_{u}
\]
is a natural bijection and we can think of $\Stable(MU(1),R_{q})_{u}$ as an
enrichment of $\HoRing(MU,R_{q})$ into $\pi_{0}$ of a space.  Indeed, the map
\[
\HoRing(MU,R_{q})\to \HoRing(MU,R_{q-1})
\]
is compatible with the
fiber sequence
\[
\to
\Stable(MU(1),R_{q})_{u}\to
\Stable(MU(1),R_{q-1})_{u}\to
\Stable(MU(1),H\vee \Sigma^{q+1}H\pi_{q}R)_{u}
\]
induced by the principal fibration constructing $R_{q}$.  We then have
a map of fiber sequences
\begin{equation}\label{eqcompfib}
\begin{gathered}
\xymatrix@-1pc{%
\ar[r]&
\EnRing[2](MU,R_{q})\ar[r]\ar[d]&
\EnRing[2](MU,R_{q-1})\ar[r]\ar[d]&
\EnRing[2](MU,H\vee \Sigma^{q+1}H\pi_{q}R)\ar[d]
\\
\ar[r]&
\Stable(MU(1),R_{q})_{u}\ar[r]&
\Stable(MU(1),R_{q-1})_{u}\ar[r]&
\Stable(MU(1),H\vee \Sigma^{q+1}H\pi_{q}R)_{u}.
}
\end{gathered}
\end{equation}

We can easily calculate the homotopy groups of $\Stable(MU(1),H\vee
\Sigma^{q+1}H\pi_{q}R)_{u}$ using the Thom isomorphism:
\begin{multline}\label{eqpihrs}
\pi_{m}\Stable(MU(1),H\vee \Sigma^{q+1}H\pi_{q}R)_{u}
\iso \\
\pi_{m}\Stable(\Sigma^{\infty}BU(1)_{+},H\vee \Sigma^{q+1}H\pi_{q}R)_{u}
\iso \tilde H^{q+1-m}(BU(1);\pi_{q}R)
\end{multline}
The next task is to understand the comparison map relating the
homotopy groups in~\eqref{eqpie2} and the homotopy groups
in~\eqref{eqpihrs}.  We prove that it is the obvious one.

\begin{lem}\label{lemobviousone}
The induced map on the homotopy groups of the base spaces
in~\eqref{eqcompfib} is the map $\tilde H^{q+3-m}(BSU;\pi_{q}R) \to \tilde
H^{q+1-m}(BU(1);\pi_{q}R)$ induced by the map
\[
\Sigma^{2} BU(1)\to \Sigma^{2}BU \to B^{2}BU\simeq BSU
\]
where $BU(1)\to BU$ is the inclusion and $\Sigma^{2}BU\to B^{2}BU$ is
the adjoint of the canonical delooping equivalence $BU\to
\Omega^{2}B^{2}BU$.
\end{lem}

\begin{proof}
The weak equivalence
\[
\EnRing[2](MU,H\vee \Sigma^{q+1}H\pi_{q}R)\simeq
\EnRing[2](\Sigma^{\infty}_{+}BU,H\vee \Sigma^{q+1}H\pi_{q}R)
\]
of Theorem~\ref{thmgenthom} is induced by a map which is just the
usual Thom isomorphism map on the underlying spectra, and so we have a
commuting diagram
\[
\xymatrix@-1pc{%
\EnRing[2](MU,H\vee \Sigma^{q+1}H\pi_{q}R)
\ar@{{}{}{}}[r]|(.475){\relax\textstyle\simeq}\ar[d]&
\EnRing[2](\Sigma^{\infty}_{+}BU,H\vee \Sigma^{q+1}H\pi_{q}R)\ar[d]&
\\
\Stable(MU,H\vee \Sigma^{q+1}H\pi_{q}R)_{u}
\ar@{{}{}{}}[r]|(.475){\relax\textstyle\simeq}\ar[d]&
\Stable(\Sigma^{\infty}_{+}BU,H\vee \Sigma^{q+1}H\pi_{q}R)_{u}\ar[d]
\\
\Stable(MU(1),H\vee \Sigma^{q+1}H\pi_{q}R)_{u}
\ar@{{}{}{}}[r]|(.475){\relax\textstyle\simeq}&
\Stable(\Sigma^{\infty}_{+}BU(1),H\vee \Sigma^{q+1}H\pi_{q}R)_{u}
}
\]
with the bottom pair of vertical maps just induced by the inclusions
$MU(1)\to MU$ and $BU(1)\to BU$.   This gives the first step in the
lemma, factoring the map in the statement through the map $\tilde
H^{q+1-m}(BU,\pi_{q}R)\to \tilde H^{q+1-m}(BU(1),\pi_{q}R)$.

We next need to bring in the equivalence in Theorem~\ref{thmbmmain}
and this requires a detour into the category of spectra lying over
$S$. If we write $\Stable_{/S}$ for the derived
space of maps in the category of spectra lying over $S$ and similarly
$\Stable_{/H}$ for the derived space of maps in the category of
spectra lying over $H$, then it is easy to see
that each of the maps
\begin{multline*}
\Stable_{/S}(\Sigma^{\infty}_{+}BU,S\vee \Sigma^{q+1}H\pi_{q}R)\to
\Stable_{/H}(\Sigma^{\infty}_{+}BU,H\vee \Sigma^{q+1}H\pi_{q}R)\to\\
\Stable(\Sigma^{\infty}_{+}BU,H\vee \Sigma^{q+1}H\pi_{q}R)_{u}
\end{multline*}
is a weak equivalence.

Let $I^{\rd}$ be the functor from spectra lying over $S$ back to
spectra that takes the homotopy fiber of the map to $S$.  Then we have
a weak equivalence
\[
\Stable_{/S}(\Sigma^{\infty}_{+}BU,S\vee \Sigma^{q+1}H\pi_{q}R)\to
\Stable(I^{\rd}\Sigma^{\infty}_{+}BU,\Sigma^{q+1}H\pi_{q}R).
\]
The relevance of this that for any augmented $E_{2}$ ring spectrum
$A$ and any spectrum $N$, the diagram
\[
\xymatrix@-1pc{%
\Stable(\Sigma^{-2}I^{\rd}B^{2}A,N)
\ar@{{}{}{}}[r]|(.475){\relax\textstyle\simeq}\ar[d]&
\EnRing[2]_{/S}(A,S\vee N)\ar[d]\\
\Stable(I^{\rd}A,N)\ar@{{}{}{}}[r]|(.475){\relax\textstyle\simeq}&
\Stable_{/S}(A,S\vee N)
}
\]
commutes, where the lefthand vertical arrow is induced by the map
$\Sigma^{2}I^{\rd}A\to I^{\rd}B^{2}A$ and the righthand vertical arrow
is the forgetful map.  The top horizontal map is the map from
\cite[7.4]{bmthh} and the fact that the diagram commutes is clear from
explicit construction given there (in the non-unital context),
cf. \cite[8.2]{bmthh} (and the discussion preceding it). In the
case of $A=\Sigma^{\infty}BU_{+}$, we use the unit $S\to
\Sigma^{\infty}_{+}BU$ to split the augmentation
$\Sigma^{\infty}_{+}BU\to S$, and then just as in the proof of
Theorem~\ref{thmbmmain} we can identify $I^{\rd}\Sigma^{\infty}_{+}BU$
as $\Sigma^{\infty}BU$ and $I^{\rd}B^{2}\Sigma^{\infty}BU_{+}$ as
$\Sigma^{\infty}B^{2}BU$.  Under these identifications
\[
\Sigma^{2}I^{\rd}\Sigma^{\infty}_{+}BU\to I^{\rd}\Sigma^{\infty}_{+}B^{2}BU
\]
becomes $\Sigma^{\infty}$ applied to the based map $\Sigma^{2}BU\to
B^{2}BU$. We then get a commuting diagram
\[
\xymatrix@-1pc{%
\Stable(\Sigma^{-2}\Sigma^{\infty}B^{2}BU,\Sigma^{q+1}H\pi_{q}R)
\ar@{{}{}{}}[r]|(.5125){\relax\textstyle\simeq}\ar[d]
&\EnRing[2](\Sigma^{\infty}_{+}BU,H\vee H\pi_{q}R)\ar[d]\\
\Stable(\Sigma^{\infty}BU,\Sigma^{q+1}H\pi_{q}R)
\ar@{{}{}{}}[r]|(.5125){\relax\textstyle\simeq}
&\Stable(\Sigma^{\infty}BU_{+},H\vee H\pi_{q}R)_{u},
}
\]
completing the proof of the lemma.
\end{proof}

As an immediate consequence of the lemma, we get the following result.

\begin{prop}
The map $\EnRing[2](MU,H\vee H\Sigma^{q+1}\pi_{q}R)\to
\Stable(MU(1),H\vee \Sigma^{q+1}H\pi_{q}R)$ is a split surjection on
homotopy groups.
\end{prop}

\begin{proof}
Applying the lemma, we can compute the induced map on homotopy groups
by computing the map on integral homology
\[
H_{*}BU(1)\to H_{*}BU\to H_{*+2}BSU
\]
and using universal coefficients.  Showing that the map on homology is
a split injection is an easy exercise using the calculation of the
Bott map (see Proposition~\ref{propbott2} below) or the edge
homomorphism in the Rothenberg-Steenrod spectral sequence
(applied twice, each spectral sequence degenerating at $E_{2}$ for
formal reasons).
\end{proof}

Finally, we can complete the proof of the inductive step.  To simplify
notation, we rewrite the diagram of fibration
sequences~\eqref{eqcompfib}
as
\[
\xymatrix{%
\ar[r]&F\ar[r]^{i}\ar[d]_{h_{F}}&E\ar[r]^{g}\ar[d]_{h_{E}}&B\ar[d]_{h_{B}}\\
\ar[r]&F'\ar[r]_{i'}&E'\ar[r]_{g'}&B'.
}
\]
We have that both base spaces $B$ and $B'$ are connected.  For each
basepoint $e$ of $E$, let $F_{e}$ denote the components of $F$ that
lie above the component of $e$, and similarly for $E'$ and $F'$.
Recall that in the long exact sequence of a fibration, at the
$\pi_{0}$ level ``exact'' means that for each $e$ in $E$,
$\pi_{0}F_{e}$ is a transitive $\pi_{1}(B,g(e))$-set with isotropy 
group at the component of $e$ (in $F_{e}$) the image of
$\pi_{1}(E,e)$. 

In our case, by
the inductive hypothesis, we have that $\pi_{1}(E,e)$ is trivial for
all $e$, and by inspection, we see that
\[
\pi_{1}(E',e')\iso
\pi_{1}(\Stable(\Sigma^{\infty}_{+}\bC P^{\infty},R_{q-1})_{u},e')
\]
is trivial for all $e'$.  It follows that for each $e$, $\pi_{0}F_{e}$
is a free transitive $\pi_{1}(B,g(e))$-set and for each $e'$,
$\pi_{0}F'_{e'}$ is a free transitive $\pi_{1}(B',g'(e'))$-set.  By
the previous proposition, $\pi_{1}(B,g(e))\to
\pi_{1}(B',g'(h_{E}(e)))$ is a surjection for every $e$ in $E$, and so
$\pi_{0}F_{e}\to \pi_{0}F'_{h_{E}(e)}$ is a surjection for every $e$
in $E$.  Letting $e$ vary over a choice of basepoint in each component
of $E$, we then see that the map
\[
\pi_{0}F\to \pi_{0}E\times_{\pi_{0}E'}\pi_{0}F'
\]
is surjective, which proves the inductive step for~(ii).  Since $\pi_{0}E\to
\pi_{0}E'$ is surjective by inductive hypothesis~(i), it follows
that $\pi_{0}F\to \pi_{0}F'$ is surjective, which proves the inductive
step for~(i).  This completes the proof of Theorem~\ref{thmMU}.

\begin{rem}\label{remMSO}
The proof for $MSO$ follows the same outline as the proof for $MU$,
taking advantage of the $2$-local equivalence between $BSO$ and $BSp$
and $MSO$ and $MSp$ (the latter attributed to \cite{Novikov} in
\cite{Stong-SC}).  Since $1/2\in \pi_{0}R$, every map of ring spectra
from $MSO$ to $R$ extends uniquely to a map of ring spectra from $MSp$
to $R$ and is determined by its restriction to a map of spectra
$MSp(1)\to R$, inducing a bijection $\EnRing[2](MSO,R)\to
\Stable(MSp(1),R)_{u}$ (cf.~\cite[7.5]{ConnerFloyd}).  We have here
that $H_{*}(B^{2}BSO;Z_{(2)})\iso H_{*}(Sp/SU;Z_{(2)})$ is torsion
free and concentrated in even degrees, and the rest of the argument
goes through as above.
\end{rem}

\section{Proof of Theorem~\ref{thmE4}}
\label{secE4}

In this section we show that not all ring spectrum maps $MU\to MU$ are
represented by $E_{4}$ ring maps.  We proceed by studying the
forgetful map from the set $\pi_{0}\EnRing(MU,MU)$ of self-maps of
$MU$ in the homotopy category of $E_{n}$ ring spectra to the set
$\HoRing(MU,MU)$ of self-maps of $MU$ in the category of ring spectra.
Although we do not obtain complete results, we do obtain enough to see
that the map is not surjective for $n\geq 4$.

We begin with a refinement of the work in the previous section.
Since the target $MU$ is an $E_{\infty}$ ring spectrum and comes with
a canonical $E_{\infty}$ ring map $MU\to MU$ (namely, the identity),
Theorem~\ref{thmThomIso} gives us a canonical weak equivalence
\[
\EnRing(MU,MU)\simeq \EnRing(\Sigma^{\infty}_{+}BU,MU)
\]
for all $n$.  Using the adjunction of \cite[IV.1.8]{mayeinf}, we can
identify $\EnRing(\Sigma^{\infty}_{+}BU,MU)$ as the derived mapping
space
\[
\EnSpace(BU,\Omega^{\infty}MU^{\times})=\EnSpace(BU,SL_{1}MU)
\]
in the category of $E_{n}$ spaces, where we regard $\Omega^{\infty}MU$
as an $E_{n}$ space via the multiplicative (rather than additive)
$E_{\infty}$ structure.  Here $SL_{1}MU$ denotes the $1$-component of
$\Omega^{\infty}MU$; since $BU$ is connected, any $E_{n}$ map must
land in the $1$-component.   As $BU$ and $SL_{1}MU$ are both
connected, the theory of iterated loop spaces gives us a weak
equivalence
\[
\EnSpace(BU,SL_{1}MU)\simeq \Top_{*}(B^{n}BU,B^{n}SL_{1}MU).
\]
Because $SL_{1}MU$ is a connected $E_{\infty}$ space, it is the zeroth
space of a connective spectrum that we denote as $sl_{1}MU$.  We then
have an identification of the homotopy groups of $\EnRing(MU,MU)$ in
terms of the cohomology theory $sl_{1}MU$. Specifically,
\begin{equation}\label{eqbnbu}
\pi_{q}\EnRing(MU,MU)\iso\pi_{q} \Top_{*}(B^{n}BU,B^{n}SL_{1}MU)
=(\widetilde{sl}_{1}MU)^{n-q}(B^{n}BU)
\end{equation}
(where tilde indicates the reduced cohomology theory),
and in particular
\[
\pi_{0}\EnRing(MU,MU)\iso (\widetilde{sl}_{1}MU)^{n}(B^{n}BU).
\]
We may further identify $(\widetilde{sl}_{1}MU)^{n}(B^{n}BU)$ as
$(\widetilde{sl}_{1}MU)^{n}(BU\six[n+2])$ when $n$ is even or
$(\widetilde{sl}_{1}MU)^{n}(U\six[n+2])$ when $n$ is odd, by Bott
periodicity.  We regard~\eqref{eqbnbu} as a refinement of
Theorem~\ref{thmPT}, as indicated in the introduction.

The parallel (now classical) theory for maps of ring spectra
$MU$ to $MU$, q.v.~\cite[II\S4]{AdamsBlueBook}, provides the identification
\begin{multline*}
\HoRing(MU,MU)\iso \HSpace(BU,SL_{1}MU) \iso \pi_{0}\Top_{*}(BU(1),SL_{1}MU)\\
=(\widetilde{sl}_{1}MU)^{0}(BU(1)),
\end{multline*}
where $\HSpace$ denotes the set of maps of $H$-spaces (maps in the
homotopy category that respect the unit and multiplication in the
homotopy category).  As both the identification of
$\pi_{0}\EnRing(MU,MU)$ and $\HoRing(MU,MU)$ in terms of reduced
$sl_{1}MU$-cohomology are induced by the Thom isomorphism for the
identity map of $MU$ together with the $(\Sigma^{\infty}$,
$\Omega^{\infty})$ adjunction, we immediately obtain the following
comparison result.

\begin{prop}\label{propbu1tobnbu}
The map
\begin{multline*}
(\widetilde{sl}_{1}MU)^{n}(B^{n}BU)\iso \pi_{0}\EnRing(MU,MU)\\
\to \HoRing(MU,MU)\iso (\widetilde{sl}_{1}MU)^{0}(BU(1))
\end{multline*}
induced by the forgetful map $\pi_{0}\EnRing(MU,MU)\to \HoRing(MU,MU)$
is the map on reduced $sl_{1}MU$ cohomology induced by the usual map
\[
\Sigma^{n}BU(1)\to \Sigma^{n}BU\to B^{n}BU.
\]
\end{prop}

When $n$ is even, we can take advantage of Bott periodicity
$B^{n}BU\simeq BU\six[n+2]$ to identify $BU(1)\to BU\six[n+2]$ as the
map induced by the Bott map, whose effect on complex oriented homology
theories is well-understood \cite[II\S12]{AdamsBlueBook} (at least
after composing with the map $BU\six[n+2]\to BU$).  Of course,
$sl_{1}MU$ is not even a ring theory, so not complex oriented, but we
can use the
Atiyah-Hirzebruch spectral sequence to obtain some information.  For
example, the integral cohomology $H^{*}(BU(1))=H^{*}(\bC P^{\infty})$
is the polynomial ring $\bZ[x]$ and in particular is in each degree
a finitely generated free abelian group and is concentrated in even
degrees.  The same is true of $\pi_{*}sl_{1}MU$, and so the
Atiyah-Hirzebruch spectral sequence has $E_{2}=E_{\infty}$, with no
extension problems, giving us a non-canonical isomorphism
\[
(\widetilde{sl}_{1}MU)^{q}(BU(1))\iso
  \bigoplus_{m>0}\tilde H^{m+q}(BU(1))\otimes \pi_{m}MU,
\]
noting that $\pi_{m}sl_{1}MU=\pi_{m}MU$ for $m>0$ whereas
$\pi_{0}sl_{1}MU=0$.  Likewise, in the case $n=2$ and $n=4$, we have
that $H^{*}B^{2}BU\iso H^{*}BSU$ is the polynomial ring
$\bZ[c_{2},c_{3},c_{4},\dotsc ]$ on Chern classes, and
$H^{*}B^{4}BU\iso H^{*}BU\six$ is a polynomial ring
$\bZ[y_{3},y_{4},y_{5},\dotsc]$ on classes in degrees $6,8,10,\dotsc$
\cite[4.7]{AHSCube}.  We then get non-canonical isomorphisms
\begin{gather*}
(\widetilde{sl}_{1}MU)^{q}(BSU)\iso
  \bigoplus_{m>0}\tilde H^{m+q}(BSU)\otimes \pi_{m}MU\\[1ex]
(\widetilde{sl}_{1}MU)^{q}(BU\six)\iso
  \bigoplus_{m>0}\tilde H^{m+q}(BU\six)\otimes \pi_{m}MU.
\end{gather*}
Up to filtration (but only up to filtration), we can identify the maps
\begin{gather*}
(\widetilde{sl}_{1}MU)^{2}(BSU)\to
(\widetilde{sl}_{1}MU)^{0}(BU(1))
\\
(\widetilde{sl}_{1}MU)^{4}(BU\six)\to
(\widetilde{sl}_{1}MU)^{0}(BU(1))
\end{gather*}
in terms of the maps $\tilde H^{*+2}(BSU)\to \tilde H^{*}(BU(1))$ and
$\tilde H^{*+4}(BU\six) \to \tilde H^{*}(BU(1))$ on ordinary
cohomology.  We now compute the maps on ordinary cohomology.

\begin{prop}\label{propbott2}
The map $H^{*+2}(BSU)\to H^{*}(BU(1))$ kills decomposable elements and
sends $c_{m+1}$ to $(-1)^{m}x^{m}$.
\end{prop}

\begin{proof}
The map clearly kills products as it is induced by the map of spaces
\[
\Sigma^{2}BU(1)\to \Sigma^{2}BU\to BSU,
\]
and products in $H^{*}(\Sigma^{2}BU(1))$ are zero.  To see where the
element $c_{m+1}$ goes, we note that the composite map
\[
\Sigma^{2}BU\to BSU\to BU
\]
is the Bott map $B$, whose effect on homology was studied in
\cite[II\S12]{AdamsBlueBook}.  We write
$H_{*}(BU)=\bZ[b_{1},b_{2},\dotsc]$, where the $b_{m}$ are the usual
generators: $b_{m}$ is the image of the usual generator of
$H_{2m}(BU(1))$ which is dual to $x^{m}\in
H^{2m}(BU(1))$.   Then on homology $B_{*}\colon H_{*}(BU)\to
H_{*}(BU)$ kills decomposable elements and sends $b_{m}$ to
$(-1)^{m}s_{m}$, where $s_{m}=q_{m}(b_{1},\dotsc,b_{m})$ and $q_{m}$
is the $m$-th Newton polynomial defined by the relationship
\[
q_{m}(\sigma_{1},\dotsc,\sigma_{m})=t_{1}^{m}+\dotsb +t_{k}^{m}
\]
for $\sigma_{j}$ the $j$-th elementary symmetric polynomial in
$t_{1},\dotsc, t_{k}$.  Then
\[
s_{m+1}=b_{1}^{m+1}+\text{terms involving }b_{j}\text{ for }j>1.
\]
On cohomology, $H^{*}(BU)\to H^{*}(BSU)$ is the
quotient by the Chern class $c_{1}$, and so we can compute the map in
the statement by means of the Bott map $B^{*}\colon H^{*}(BU)\to H^{*}(BU)$.
Using the Kronecker pairing of homology with cohomology, we see that
\begin{align*}
\langle B^{*}c_{m+1},b_{m}\rangle
&= \langle c_{m+1},B_{*}b_{m}\rangle\\
&= \langle c_{m+1},(-1)^{m}s_{m+1}\rangle\\
&= \langle c_{m+1},(-1)^{m}b_{1}^{m+1}\rangle=(-1)^{m}
\end{align*}
since $c_{m+1}$ is the dual of $b_{1}^{m+1}$ in the monomial basis of the
$b_{m}$'s.
\end{proof}

\begin{prop}\label{propbott4}
The map $H^{*+4}(BU\six)\to H^{*}(BU(1))$ kills decomposable elements
and sends the polynomial generator $y_{m+2}$ in $H^{2m+4}(BU\six)$ to
\[
\begin{cases}
(-1)^{p^{t}}p^{t-1}x^{p^{t}-1}&m+1=p^{t}\text{ for some prime }p, t>0\\
(-1)^{m+1}(m+1)x^{m}&\text{otherwise}
\end{cases}
\]
\end{prop}

\begin{proof}
As in the proof of the previous proposition, the map clearly kills
decomposables, and we approach the problem using the Bott map
$B^{2}\colon \Sigma^{4}BU\to BU$.  Since the Bott map on homology
$B_{*}$ kills decomposables and $B_{*}b_{m}=(-1)^{m}s_{m+1}$, using
\[
s_{m+1}=(-1)^{m}(m+1)b_{m+1}+\text{decomposables}
\]
we see that $B_{*}^{2}b_{m}=(-1)^{m+1}(m+1)s_{m+2}$ and that the
composite map $H^{*+4}(BU)\to H^{*}(BU(1))$ takes $c_{m+2}$ to
$(-1)^{m+1}(m+1)x^{m}$.  Unlike in the previous proposition, the map
$H^{*}(BU)\to H^{*}(BU\six)$ is not onto.  By \cite[4.6]{AHSCube}, for
$m+1\neq p^{t}$, we can take the generator in $H^{2(m+2)}(BU\six)$ to
be the image of $c_{m+2}$.  By \cite[4.5]{AHSCube}, for $m+1=p^{t}$,
the image of $c_{m+2}$ is up to decomposables $p$ times a generator in
$H^{2(m+2)}(BU\six)$. (It is $p^{1}$ times a generator rather than
some higher power of $p$ since the trangressive elements $u_{k}$ (in
the notation of \cite[4.5]{AHSCube}) all have non-trivial Bocksteins.)
This completes the proof.
\end{proof}

In the $n=4$ case, we have that the maps $H^{2m+4}(BU\six)\to
H^{2m}(BU(1))$ are surjective for $m=1$ and $m=2$.  This says that for
any $a_{1}\in \pi_{2}MU$ and $a_{2}\in \pi_{4}MU$, there exist $E_{4}$
ring maps $MU\to MU$ whose coordinates are of the form
\[
x+a_{1}x^{2}+\dotsb \qquad \text{and}\qquad x+a_{2}x^{3}+\dotsb,
\]
but because of the filtration issue above, we cannot be sure exactly
which coordinates of these forms represent $E_{4}$ ring maps without
further work.  On the
other hand, the map $H^{2m+4}(BU\six)\to H^{2m}(BU(1))$ is not
surjective for $m=3$ but has image divisible by $2$.  This has the
consequence that if we look at any map of ring spectra
$f\colon MU\to MU$ corresponding to a coordinate of the form
\[
x+a_{3}x^{4}+\dotsb
\]
where $a_{3}\in \pi_{6}MU$ is not divisible by $2$, then $f$ cannot be
represented by an $E_{4}$ ring map $MU\to MU$.  (Similar arguments can
obviously be made at other primes.)

\section{Proofs for Section~\ref{secmodel}}\label{secpfmodel}

In this section we prove Theorems~\ref{thmOpModCat}--\ref{thmOpCof} and
construct the model structure in Theorem~\ref{thmOpComp2}.  We base
our approach on \cite[\S11--12]{EM1}, which worked in the context of
simplicial sets, but which generalizes to the current context.  For
convenience and to make this section more self-contained for future
reference, we restate (and generalize) the results as
Theorems~\ref{thmgenmodel}, \ref{thmgenchange}, and~\ref{thmgencell}
below.

Because we have already proved the theorems in Section~\ref{secmodel}
that involve functors between different categories of spectra, we can
now work with a single model category of spectra throughout this section.
We let $\Spectra$ denote one of the following
model categories:
\begin{enumerate}
\item The category $\Sigma_{*}\aS$ of symmetric
spectra with its positive stable model structure \cite[\S14]{MMSS}.
\item The category $\aI\aS$ of orthogonal spectra with its
positive model structure \cite[\S14]{MMSS}.
\item The category $\EKMM$ of EKMM
$S$-modules with its standard model structure \cite[VII\S4]{ekmm}.
\item The category $\LMS[\bL]$ of EKMM $\bL$-spectra with its standard
model structure \cite[VII\S4]{ekmm}.
\end{enumerate}
(We extend the convention used throughout the paper that the
unmodified word ``spectrum'' means precisely an object of $\Spectra$.)
We regard the category $\Spectra$ as a cofibrantly generated model
category in its standard way, and in the arguments below we use $I$ to
denote the standard set of
generating cofibrations and $J$ to denote the standard set of
generating acyclic
cofibrations.

We will actually prove mild generalizations of the theorems of
Section~\ref{secmodel} partly because
the extra generality may be useful in future papers, but mainly
because the proofs require the extra generality anyway.  In
Section~\ref{secmodel}, we worked in the context of an operad $\oO$ of
(unbased) spaces; here we let $\oO$ be an operad of based spaces or an
operad in $\Spectra$.  Indeed, for $\oO$ an operad in unbased spaces,
the category $\Spectra[\oO]$ is the same as the category of algebras
over the operad $\oO_{+}$ of based spaces, and for the true symmetric
monoidal categories of spectra, it is isomorphic (not just equivalent)
to the category of algebras over the operad $\oO_{+}\sma S$ in
$\Spectra$.  In the category of EKMM $\bL$-spectra (which we needed
for the work involving the Thom isomorphism), operads in $\Spectra$ do
not generalize operads in spaces.  With this generalization in mind,
we have written the statements and arguments below in the based
context: In what follows, $\oO$ denotes either an operad in $\Spectra$
or an operad in based spaces.

We now need to prove three theorems generalizing the statements in
Section~\ref{secmodel}.  The first establishes the model structures,
proving Theorem~\ref{thmOpModCat} and finishing the proof of
Theorem~\ref{thmOpComp2}.  (We also used it in the proof of
Theorem~\ref{thmRModMod}.) 

\begin{thm}\label{thmgenmodel}
Let $\oO$ be an operad.  Then
the category $\Spectra[\oO]$ of $\oO$-algebras in $\Spectra$  is a
topological
closed model category with fibrations and weak equivalences created in
$\Spectra$.
\end{thm}

The next proves Theorem~\ref{thmChangeOp}.   In the statement $\rO$
denotes the free $\oO$-algebra functor, which is
\begin{align*}
\rO X &= \bigvee_{n\geq 0} \oO(n)\sma_{\Sigma_{n}} X^{(n)}\\
&= (\oO(0)\sma S) \vee (\oO(1)\sma X) \vee
(\oO(2)\sma X\sma X)/\Sigma_{2}\vee \dotsb.
\end{align*}
when $\oO$ is an operad of based spaces or an operad in $\Spectra$
when $\Spectra$ is one of the true symmetric monoidal categories.
Here we have written $\sma$ both for the smash product in
$\Spectra$ and the smash product of a based space with a spectrum, and
we have used parenthetical exponent
\[
X^{(n)}=X\sma \dotsb \sma X
\]
as an abbreviation for smash powers.  In the case when $\Spectra$ is
the category of EKMM $\bL$-spectra and $\oO$ is an operad in
$\Spectra$, the free functor needs the following modifications: In
homogeneous degree 0, we need to use $\oO(0)\rhd S\iso \oO(0)$ in place
of $\oO(0)\sma S$ and in homogeneous degree 1, we need to use
$\oO(1)\lhd X$ in place of $\oO(1)\sma X$, where $\lhd$ and $\rhd$ denote
the one-sided unital products of \cite[XIII.1.1]{ekmm}.  (In general,
in the case of EKMM $\bL$-spectra, we need to use a unital product
$\lhd$, $\rhd$, or $\star$ in place of a smash product whenever one or
both factors comes with a structure map from $S$; in what follows, we
refer to this as ``the usual modifications''.)

\begin{thm}\label{thmgenchange}
Let $\phi\colon \oO\to \oO'$ be a
map of operads.  The pushforward
(left Kan extension) and pullback functors
\[
L_{\phi}\colon \Spectra [\oO]\tofrom \Spectra[\oO']\noloc R_{\phi}
\]
form a Quillen adjunction, which is a Quillen equivalence if (and only
if) the induced map on free algebras
\[
\rO X\to \rO' X
\]
is a weak equivalence for all cofibrant objects $X$.
\end{thm}

To deduce Theorem~\ref{thmChangeOp}, we need to show that in the
context of operads of unbased spaces, $\phi$
induces a weak equivalence on free algebras as in the statement if and
only if each $\phi(n)$ is a (non-equivariant) stable equivalence.
The ``if'' direction is a straight-forward generalization of
\cite[15.5]{MMSS} (in the case of symmetric spectra and orthogonal
spectra) or \cite[III.5.1]{ekmm} (in the case of EKMM $S$-modules or
$\bL$-spectra) that follows by essentially the same argument.  The
``only if'' direction follows by taking $X$ to be a wedge of cofibrant
$0$-spheres $\bigvee S^{0}_{c}$; for a wedge of $n$ or more, $\rO X$
and $\rO' X$ contain
\[
\oO(n)_{+}\sma (S^{0}_{c})^{(n)}
\qquad \text{and}\qquad
\oO'(n)_{+}\sma (S^{0}_{c})^{(n)}
\]
(respectively) as wedge summands.  With an eye toward more generality,
we offer the following additional remark on the criterion in the
theorem above in the case when $\oO$ is an operad in $\Spectra$.

\begin{rem}
The criterion that $\rO X\to \rO' X$ is
a weak equivalence for every cofibrant object $X$ is satisfied in
particular in the following cases.
\begin{enumerate}
\item In the case when $\Spectra$ is the category of symmetric spectra
or orthogonal spectra, the criterion is satisfied whenever each map
$\phi(n)\colon \oO(n)\to \oO'(n)$ is a (non-equivariant) weak
equivalence.  This follows from the observation that the proof of
\cite[15.5]{MMSS} still works when (in the notation there) $X$ (our
$\oO(n)$) also
has a $\Sigma_{n}$ action: The $\Sigma_{n}$ action remains free on
$O(q)$ (or $\Sigma_{q}$).  Induction up the cellular filtration of
$E\Sigma_{n}$ shows that when $X$ is cofibrant,
\[
E\Sigma_{n+}\sma_{\Sigma_{n}}(\oO(n)\sma X^{(n)})
\]
preserves (non-equivariant) weak equivalences in (equivariant) maps of
$\oO(n)$.
\item In the case when $\Spectra$ is EKMM $S$-modules, the criterion
is satisfied whenever each $\phi(n)$ is a (non-equivariant) weak
equivalence and each $\oO(n)$ and $\oO'(n)$ has underlying
non-equivariant object in the class $\bar\aE$ of \cite[VII.6.4]{ekmm}
(or Basterra's generalization $\bar \aF$ of \cite[9.3]{bmthh}), or
more generally, the closure of $\aE$ (or $\aF$) under also the
additional operation of smash product with a based space.  The proof
is essentially the same as the proof of \cite[III.5.1]{ekmm}.  In the
case of EKMM $\bL$-spectra it also works for the analogous class
(where we allow omitting $S\sma_{\oL}(-)$).
\end{enumerate}
Both cases include in particular operads of the form $\oO\sma S$ when
$\oO$ is an operad of based spaces.
For based spaces with non-degenerate basepoints (e.g., disjoint
basepoints), a map $X\to X'$ is a stable equivalence if and only if
$X\sma S\to X'\sma S$ is a weak equivalence; however, the same is not
necessarily true for based spaces with degenerate basepoints, so some
caution is in order when applying the remarks above in the context of
operads of based spaces.
\end{rem}

Finally, the third result generalizes Theorem~\ref{thmOpCof}.  To
state it, we need to generalize the hypothesis on $\oO$.  We use the
following terminology.

\begin{defn}\label{defsigmafree}
An operad (or symmetric sequence) $\oO$ of based spaces is a
\term{$\Sigma$-free cell retract} if for each $n>0$, $\oO(n)$ is the retract
of a free based $\Sigma_{n}$-cell complex.  An operad (or symmetric
sequence) $\oO$ in $\Spectra$ is a \term{$\Sigma$-free cell retract}
if each $\oO(n)$ is equivariantly the retract of a $\Sigma_{n}$-equivariant
spectrum built equivariantly as a complex with cells of the form
\[
\Sigma_{n+}\sma X\to \Sigma_{n+}\sma Y
\]
where $X\to Y$ is a wedge of maps in $I$ and/or maps of the form
\[
S^{j-1}_{+}\sma S\to D^{j}_{+}\sma S
\]
where $S^{j-1}\to D^{j}$ is the inclusion of the boundary of the
standard $j$-dimensional disk (or the inclusion of the empty set in the
one-point space for $j=0$).  Note that there is no
condition on $\oO(0)$.
\end{defn}

The previous definition is adapted for ease of use in the proof of the
following theorem that directly generalizes Theorem~\ref{thmOpCof}.

\begin{thm}\label{thmgencell}
Assume either that $\oO$ is an operad of based spaces or that
$\Spectra$ is symmetric spectra, orthogonal spectra, or EKMM
$S$-modules.  If $\oO$ is a $\Sigma$-free cell retract, then every
cofibrant object in $\Spectra[\oO]$ is cofibrant
in the category of $\Spectra$ under $\rO(*)$.
\end{thm}

Before going on to the proofs, we make the following remark about
generalizing Theorem~\ref{thmOpComp1}.

\begin{rem}
In order to keep $\Spectra$ fixed, we did not restate
Theorem~\ref{thmOpComp1} in this section, nor do we prove it below;
nevertheless, Theorem~\ref{thmOpComp1} does generalize to the case
when $\oO$ is an operad in the domain category of the left adjoint,
provided we add the hypothesis that the unit of the adjunction is a
weak equivalence for $\rO X$ for all cofibrant $X$.
The proof follows the same outline as the proof of
Theorem~\ref{thmgenchange}.
\end{rem}

We now move on to the proofs.  We fix $\oO$ an operad of based spaces
or an operad in $\Spectra$, and we write $\rO I$ and $\rO J$ for the
sets of maps in
$\Spectra[\oO]$ obtained by applying $\rO$ to $I$ and $J$, respectively
(where as indicated above, $I$ and $J$ are the canonical sets of
generating cofibrations and generating acyclic cofibrations,
respectively).  According to \cite[5.13]{MMSS}, to show that
$\Spectra[\oO]$ is a cofibrantly generated topological model category with
generating cofibrations $\rO I$ and generating acyclic cofibrations
$\rO J$, it suffices to prove the following lemma.

\begin{lem}\label{lemModel}
Let $C$ be an object in $\Spectra[\oO]$.
\begin{enumerate}
\item For $A\to B$ any coproduct (in $\Spectra[\oO]$) of maps in $\rO
I$ and any map of $\oO$-algebras $A\to C$,
the map $C\to C\amalg_{A}B$ from $C$ to the pushout in $\Spectra[\oO]$ is
an $h$-cofibration in $\Spectra$.
\item For $A\to B$ any coproduct (in $\Spectra[\oO]$) of maps in $\rO
J$ and any map of $\oO$-algebras $A\to C$,
the map $C\to C\amalg_{A}B$ from $C$ to the pushout in $\Spectra[\oO]$ is a
weak equivalence in $\Spectra$.
\end{enumerate}
\end{lem}

In the statement above, an $h$-cofibration is a map $X\to Y$
satisfying the homotopy extension property, or in other words, such
that the map $Y\cup_{X}(X\sma I_{+})\to Y\sma I_{+}$ admits a
retraction.

The key to proving the lemma is understanding pushouts in $\Spectra[\oO]$
of the form $C\to C\amalg_{\rO X}\rO Y$.  We will show that the
underlying spectrum has a filtration induced by powers of $Y$.
To construct this, we use the universal enveloping operad of
\cite[8.3]{bm1} and \cite[\S12]{EM1}.

\begin{cons}\label{defuniv}
For an $\oO$-algebra $C$, define $U_{\oO}C(n)$ to be the coequalizer in
$\Spectra$
\[
\xymatrix{%
\bigvee_{k} \oO(n+k)\sma_{\Sigma_{k}}(\rO C)^{(k)}
\ar@<-.75ex>[r]\ar@<.75ex>[r]
&\bigvee_{k} \oO(n+k)\sma_{\Sigma_{k}}C^{(k)}
\ar[r]
&U_{\oO}C(n)
}
\]
(with the usual modifications when $\Spectra$ is the category of EKMM
$\bL$-spectra).  Here one map is induced by the action $\rO C\to C$
and the other by the operadic multiplication.  The spectra
$U_{\oO}C(-)$ form an operad in
$\Spectra$ with $\Sigma_{n}$ action on $U_{\oO}C(n)$ induced by the
unused $\Sigma_{n}$ action on $\oC(k+n)$, identity $S\to U_{\oO}C(1)$
induced by the identity of $\oO$ and operadic multiplication induced
by the operadic multiplication of $\oO$.
\end{cons}

(In what follows, we never actually use the operadic multiplication, just the
identity and equivariance.)

An easy check of universal
properties shows that the coproduct of $\oO$-algebras $C\amalg \rO Y$ has
\[
\bigvee_{n}U_{\oO}C(n)\sma_{\Sigma_{n}}  Y^{(n)}
\]
as its underlying spectrum when $\Spectra$ is one of the true
symmetric monoidal categories of spectra.  For EKMM $\bL$-spectra, we
have the usual modification discussed above
\[
C\amalg \rO Y = (U_{\oO}C(0)) \vee (U_{\oO}C(1)\lhd Y) \vee
\bigvee_{n>1}(U_{\oO}C(n)\sma_{\Sigma_{n}}  Y^{(n)})
\]
when $\oO$ is an operad in $\Spectra$, but when $\oO$ is an operad of
based spaces, we have a slightly different formula.  (Here for clarity
we temporarily break our convention and write $\sma_{\oL}$ for the smash
product of $\bL$-spectra, reserving $\sma$ for the smash product with
a space.)  For $\oO$ an operad of based spaces, the summands $k=0$
in Construction~\ref{defuniv} induce a map of
$\Sigma_{n}$-equivariant $\bL$-spectra $\oO(n)\sma S\to U_{\oO}C$.
For $Z$ an $\bL$-spectrum, define
$U_{\oO}C(n)\lhd_{\oO(n)}Z$ to be the following pushout.
\[
\xymatrix{%
\oO(n)\sma S\sma_{\oL}Z\ar[r]\ar[d]&U_{\oO}C(n)\sma_{\oL}Z\ar[d]\\
\oO(n)\sma Z\ar[r]&U_{\oO}C(n)\lhd_{\oO(n)}Z
}
\]
We note that both vertical maps are weak equivalences, as they become
isomorphisms after smashing with $S$ (applying $S\sma_{\oL}(-)$).
If $H<\Sigma_{n}$ and $Z$ is an $H$-equivariant $\bL$-spectrum, then
$U_{\oO}C(n)\lhd_{\oO(n)}Z$ has a right action of $H$ from
$U_{\oO}C(n)$ (and $\oO(n)$) and a left action from $Z$, and we let
$(U_{\oO}C(n)\lhd_{\oO(n)}Z)_{H}$ denote the coequalizer of these actions.
Then the universal property of the coproduct gives us
\[
C\amalg \rO Y = \bigvee_{n} (U_{\oO}C(n)\lhd_{\oO(n)}Y^{(n)})_{\Sigma_{n}}.
\]
In general, when $\Spectra$ is EKMM $\bL$-spectra and $\oO$ is an
operad of based spaces, whenever we encounter a formula involving the
universal enveloping operad, we must replace $U_{\oO}C(n)\sma_{H}Z$ with
$(U_{\oO}C(n)\lhd_{\oO(n)}Z)_{H}$; we add this to our list of ``usual
modifications'' for the case of EKMM $\bL$-spectra.

The discussion above gives us a (split) filtration on $C\amalg \rO Y$.
To get the filtration on $C\amalg_{\rO X}\rO Y$, we also need the
following construction from \cite[\S12]{EM1}.

\begin{cons}
For $g\colon X\to Y$ a map in $\Spectra$, define $Q^{n}_{i}(g)$ inductively as
follows. Let $Q^{n}_{0}(g)=X^{(n)}$ and for $i>0$, define
$Q^{n}_{i}(g)$ to be the pushout
\[
\xymatrix{%
\Sigma_{n+}\sma_{\Sigma_{n-i}\times \Sigma_{i}}
X^{(n-i)}\sma Q^{i}_{i-1}(g)\ar[r]\ar[d]
&\Sigma_{n+}\sma_{\Sigma_{n-i}\times \Sigma_{i}}
X^{(n-i)}\sma Y^{(i)}\ar[d]
\\
Q^{n}_{i-1}(g)\ar[r]&Q^{n}_{i}(g)
}
\]
\end{cons}

The basic idea is that (when $g$ is an inclusion) $Q^{n}_{i}(g)$ is
the $\Sigma_{n}$-equivariant
subspectrum of $Y^{(n)}$ with $i$ factors of $Y$ and $n-i$
factors of $X$.  Mainly we
need $Q^{n}_{n-1}(g)$ and we see that
$Y^{(n)}/Q^{n}_{n-1}(g)\iso (Y/X)^{(n)}$.  Just as in
\cite[12.6]{EM1}, we have the following proposition.

\begin{prop}\label{propfiltpush}
Let $C$ be an $\oO$-algebra and let $g\colon X\to Y$ be a map in
$\Spectra$.  For any map $X\to C$ in $\Spectra$, let
$Fil_{0}(C,g)=C$ and
inductively define $Fil_{n}(C,g)$ to be the pushout
\[
\xymatrix{%
U_{\oO}C(n)\sma_{\Sigma_{n}}  Q^{n}_{n-1}(g)\ar[r]\ar[d]
&U_{\oO}C(n)\sma_{\Sigma_{n}}  Y^{(n)}\ar[d]\\
Fil_{n-1}(C,g)\ar[r]&Fil_{n}(C,g)
}
\]
(with the usual modifications when $\Spectra$ is the category of EKMM
$\bL$-spectra).
Then $\colim Fil_{n}(C,g)$ is the underlying spectrum of the
pushout $C\amalg_{\rO X}\rO Y$ in $\Spectra[\oO]$.
\end{prop}

\begin{proof}
Using the constructions of $U_{\oO}C(n)$ and $Q^{n}_{n-1}(g)$, and
commuting colimits, we see that $\colim Fil_{n}(C,g)$ can be identified as
the coequalizer of the following pair of arrows
\begin{multline*}
\bigvee_{i,k} \oO(k+n)\sma_{\Sigma_{k}\times \Sigma_{n-i}\times \Sigma_{i}}(\rO C)^{(k)}\sma X^{(n-i)}\sma Y^{(i)}\\[-1em]
\xymatrix@C=1pc{\ar@<-.75ex>[r]\ar@<.75ex>[r]&}
\bigvee_{k} \oO(k+n)\sma_{\Sigma_{k}\times \Sigma_{n}}C^{(k)}\sma Y^{(n)}
\end{multline*}
(with the usual modifications when $\Spectra$ is the category of EKMM
$\bL$-spectra)
where one map is induced by the action map $\rO C\to C$ and the map
$g\colon X\to Y$ and the other is induced by the
operadic multiplication on $\oO$ and the given map $X\to C$.  Comparing
universal properties, the coequalizer above is easily identified as
the pushout $C\amalg_{\rO X}\rO Y$ in $\Spectra[\oO]$.
\end{proof}

Proposition~\ref{propfiltpush} is what we need to prove
Lemma~\ref{lemModel} and therefore complete the proof of
Theorem~\ref{thmgenmodel}.

\begin{proof}[Proof of Lemma~\ref{lemModel}]
In both parts, we let $g\colon X\to Y$ be a wedge
of maps in $I$ (for~(i)) or $J$ (for~(ii)) such that $A\to B$ is $\rO
g$.

For (i), by Proposition~\ref{propfiltpush}, it suffices to see that
each map $Fil_{n-1}(C,g)\to Fil_{n}(C,g)$ is an $h$-cofibration (for $n>1$),
and for this it suffices to show that each map
\begin{equation}\label{eqpflemModel}\tag{*}
U_{\oO}C(n)\sma_{\Sigma_{n}}  Q^{n}_{n-1}(g)\to
U_{\oO}C(n)\sma_{\Sigma_{n}}  Y^{(n)}
\end{equation}
is an $h$-cofibration
(with the usual modification when $\Spectra$ is the category of EKMM
$\bL$-spectra).  In the case when $\Spectra$ is symmetric or
orthogonal spectra, $g$ is a wedge of maps of the form
\[
F_{i}S^{j-1}_{+}\to F_{i}D^{j}_{+}
\]
where $S^{j-1}\to D^{j}$ is the inclusion of the boundary into the
standard $j$-dimensional disk.  Then $Y^{(n)}$ becomes the wedge of
$\Sigma_{n}$-equivariant spectra of the form
\begin{multline*}
\Sigma_{n+}\sma_{\Sigma_{m_{1},\dotsc,m_{k}}}
((F_{i_{1}}D^{j_{1}}_{+})^{(m_{1)}}\sma \dotsb
\sma (F_{i_{k}}D^{j_{k}}_{+})^{(m_{k)}})\\
\iso \Sigma_{n+}\sma_{\Sigma_{m_{1},\dotsc ,m_{k}}}
F_{i}( (D^{j_{1}})^{m_{1}}\times \dotsb \times (D^{j_{k}})^{m_{k}})_{+}
\end{multline*}
where $m_{1}+\dotsb +m_{k}=n$ and $i:=i_{1}+\dotsb +i_{k}$,
$\Sigma_{m_{1},\dotsc,m_{k}}:=\Sigma_{m_{1}}\times \dotsb \times
\Sigma_{m_{k}}$.  We can then identify $Q^{n}_{n-1}(g)$ as the wedge
$\Sigma_{n}$-equivariant spectra of the form
\[
\Sigma_{n+}\sma_{\Sigma_{m_{1},\dotsc ,m_{k}}}
F_{i}\partial ( (D^{j_{1}})^{m_{1}}\times \dotsb \times (D^{j_{k}})^{m_{k}})_{+}
\]
and the map $Q^{n}_{n-1}(g)\to Y^{(n)}$ as the induced by the boundary
inclusions.  By inspection, this is a $\Sigma_{n}$-equivariant
$h$-cofibration, and it then follows that the map \eqref{eqpflemModel}
is an $h$-fibration.  The case of EKMM $S$-modules and $\bL$-spectra
is analogous.

For~(ii), the case of EKMM $S$-modules and $\bL$-spectra is trivial as
the map $X\to Y$ is the inclusion of a deformation retraction, and
therefore, the map $\rO X\to \rO Y$ is the inclusion of a deformation
retract in the category $\Spectra[\oO]$; it follows that $C\to
C\amalg_{A}B$ is the inclusion of a deformation retract in the
category $\Spectra[\oO]$ and therefore a homotopy equivalence.  For the
case of symmetric spectra or orthogonal spectra, applying
Proposition~\ref{propfiltpush}, it suffices to show that the map
$Fil_{n-1}(C,g)\to Fil_{n}(C,g)$ is a stable equivalence for all
$n\geq 1$.  The argument above
shows the map is an $h$-cofibration.  Its cofiber
is
\begin{equation}\label{eqpflemModel2}\tag{**}
U_{\oO}C(n)\sma_{\Sigma_{n}} (Y/X)^{(n)}.
\end{equation}
Since $Y/X$ is positive cofibrant and stably equivalent to the trivial
spectrum, it follows from~\cite[15.5]{MMSS} that~\eqref{eqpflemModel2}
is weakly equivalent to the trivial spectrum, and hence
that $Fil_{n-1}(C,g)\to Fil_{n}(C,g)$ is a weak equivalence.
\end{proof}

Proposition~\ref{propfiltpush} is also all we need for the proof of
Theorem~\ref{thmgenchange}.

\begin{proof}[Proof of Theorem~\ref{thmgenchange}]%
Since in both algebra categories fibrations and weak equivalences are
created in $\Spectra$ and since the right adjoint does not change the
underlying spectrum, the adjunction is a Quillen adjunction.
Now assume that $\phi$ induces a weak equivalence $\rO X\to \rO' X$
for every cofibrant object $X$.  To see that the
adjunction is a Quillen equivalence, it suffices to show that for a
cofibrant $\oO$-algebra $C$, the unit map $C\to R_{\phi}L_{\phi}C$ is
a weak equivalence.  Without loss of generality, we can assume that
$C$ is an $\rO I$-cell complex.  Then $C=\colim C_{k}$, where
$C_{0}=\rO(*)$
and $C_{k}=C_{k-1}\amalg_{A_{k}}B_{k}$ where
$A_{k}\to B_{k}$ is a coproduct of maps in $\rO I$, or equivalently,
is $\rO g_{k}$ for $g_{k}\colon X_{k}\to Y_{k}$ a wedge of maps in
$I$.  We have that $L_{\phi}C_{0}=\rO'(*)$ and by hypothesis
the map $\rO(*)\to \rO'(*)$ is a weak
equivalence.  Likewise, $C_{1}=\rO (Y_{1}/X_{1})$,
$L_{\phi}C_{1}=\rO'(Y_{1}/X_{1})$ and the unit map is a weak
equivalence. This shows that for any $\rO
I$-cell complex built in $0$ stages or $1$ stage, the unit map is a
weak equivalence.  Assume by induction that for any $\rO I$-cell
complex built in $k$ or fewer stages, the unit map is a weak
equivalence, and consider $C_{k+1}=C_{k}\amalg_{\rO X}\rO Y$ for
$g\colon X\to Y$.  Proposition~\ref{propfiltpush} writes $C_{k+1}$ as
$\colim Fil_{n}(C_{k},g)$ and similarly $L_{\phi}C_{k+1}$ is $\colim
Fil_{n}(L_{\phi}C_{k},g)$ constructed using the operad $\oO'$.  The
proof of Lemma~\ref{lemModel}) showed that this is a filtration by
$h$-cofibrations, and we note that the associated graded spectra are
\[
\bigvee U_{\oO}C_{k}(n)\sma_{\Sigma_{n}} (Y/X)^{(n)}
\qquad \text{and}\qquad
\bigvee U_{\oO'}(L_{\phi}C_{k})(n)\sma_{\Sigma_{n}} (Y/X)^{(n)}
\]
(with the usual modifications when $\Spectra$ is the category of EKMM
$\bL$-spectra).
These naturally form algebras:
The first is the $\oO$-algebra $C_{k}\amalg \rO(Y/X)$ and the second
is the $\oO'$-algebra
\[
L_{\phi}C_{k}\amalg \rO'(Y/X)=L_{\phi}(C_{k}\amalg \rO(Y/X))
\]
(with coproducts taken in the appropriate algebra category $\Spectra[\oO]$
and $\Spectra[\oO']$, respectively).  As $C_{k}\amalg (Y/X)$ is an
$\oO$-algebra that can
be built in $k$ or fewer stages (as $k\geq 1$), it follows that the
unit map
\[
C_{k}\amalg \rO(Y/X)\to
R_{\phi }L_{\phi}(C_{k}\amalg \rO(Y/X))
\]
is a weak equivalence, which shows that each map on quotients
\[
U_{\oO}C_{k}(n)\sma_{\Sigma_{n}} (Y/X)^{(n)}\to
U_{\oO'}L_{\phi}C_{k}(n)\sma_{\Sigma_{n}} (Y/X)^{(n)}.
\]
is a weak equivalence and shows that each map
\[
Fil_{n}(C_{k},g)\to Fil_{n}(L_{\phi}C_{k},g)
\]
is a weak equivalence.  It follows that $C_{k+1}\to
R_{\phi}L_{\phi}C_{k+1}$ is a weak equivalence.  Finally, $C=\colim
C_{k}$ is the colimit of a sequence of $h$-cofibrations as is
$L_{\phi}C=\colim L_{\phi}C_{k}$, and so the unit $C\to
R_{\phi}L_{\phi}C$ is a weak equivalence.
\end{proof}

Finally, we need to prove Theorem~\ref{thmgencell}.  For this we need
a slight generalization of Proposition~\ref{propfiltpush} that handles
the construction of the universal enveloping operad.

\begin{prop}\label{propfiltpush2}
Let $C$ be an
$\oO$-algebra and let $g\colon X\to Y$ be a map in $\Spectra$. For a
map $X\to C$ in $\Spectra$, let
$Fil_{0}(U_{\oO}C,g)(m)= U_{\oO}C(m)$ and inductively define
$Fil_{n}(U_{\oO}C,g)(m)$ as the pushout
\[
\xymatrix{%
U_{\oO}C(m+n)\sma_{\Sigma_{n}}  Q^{m}_{m-1}(g)\ar[r]\ar[d]
&U_{\oO}C(m+n)\sma_{\Sigma_{n}}  Y^{(n)}\ar[d]\\
Fil_{n-1}(U_{\oO}C,g)(m)\ar[r]&Fil_{n}(U_{\oO}C,g)(m)
}
\]
(with the usual modifications when $\Spectra$ is the category of EKMM
$\bL$-spectra).
Then $\colim_{n} Fil_{n}(U_{\oO}C,g)(m)$ is the underlying
$\Sigma_{n}$-equivariant spectrum of
$U_{\oO}(C\amalg_{\rO X}\rO Y)(m)$.
\end{prop}

\begin{proof}
As in the proof of Proposition~\ref{propfiltpush}, unwinding the
definitions and interchanging colimits
identifies $\colim_{n} Fil_{n}(U_{\oO}C,g)(m)$ as the
coequalizer
\begin{multline*}
\bigvee_{i,k} \oO(k+n+m)\sma_{\Sigma_{k}\times \Sigma_{n-i}\times \Sigma_{i}}(\rO C)^{(k)}\sma X^{(n-i)}\sma Y^{(i)}\\[-1em]
\xymatrix@C=1pc{\ar@<-.75ex>[r]\ar@<.75ex>[r]&}
\bigvee_{k} \oO(k+n+m)\sma_{\Sigma_{k}\times \Sigma_{n}}C^{(k)}\sma
Y^{(n)}
\to \colim_{n} Fil_{n}(U_{\oO}C,g)(m)
\end{multline*}
(with the usual modifications when $\Spectra$ is the category of EKMM
$\bL$-spectra).
We can rewrite this as the coequalizer
\[
\xymatrix@C=1pc{%
U_{\oO}(\rO((\rO C) \vee X \vee Y))(m)\ar@<-.75ex>[r]\ar@<.75ex>[r]
&U_{\oO}(\rO(C\vee Y))(m)\ar[r]&\colim_{n} Fil_{n}(U_{\oO}C,g)(m)
}
\]
which is the universal enveloping operad construction $U_{\oO}(-)(m)$
applied to the reflexive coequalizer
\[
\xymatrix@C=1pc{%
\rO((\rO C) \vee X \vee Y)\ar@<-.75ex>[r]\ar@<.75ex>[r]
&\rO(C\vee Y)\ar[r]&C\amalg_{\rO X}\rO Y.
}\qedhere
\]
\end{proof}

\begin{proof}[Proof of Theorem~\ref{thmgencell}]%
First consider the case where $\Spectra$ is one of the true symmetric
monoidal categories of spectra.  It suffices to prove that $\rO
I$-cell complexes are cofibrant in $\Spectra$ under $\rO(*)$.  Let $C$
be an $\rO I$-cell complex; then $C=\colim C_{k}$ with $C_{0}=\rO(*)$
and $C_{k+1}=C_{k}\amalg_{\rO X_{k}}\rO Y_{k}$ for $X_{k}\to Y_{k}$ a
wedge of maps in $I$. It therefore suffices to show that each map
$C_{k}\to C_{k+1}$ is a cofibration in $\Spectra$.  Since
$U_{\oO}(C_{0})(n)=\oO(n)\sma S$, by hypothesis $U_{\oO}(C_{0})$ is a
$\Sigma$-free cell retract.  Assume by induction that $U_{\oO}C_{k}$
is a $\Sigma$-free cell retract. The argument in Lemma~\ref{lemModel}
generalizes to show that each map $Q^{n}_{n-1}(g_{k})\to Y^{(n)}_{k}$
is a (non-equivariant) cofibration between cofibrant objects and each
map
\[
U_{\oO}C_{k}(n+m)\sma_{\Sigma_{n}} Q^{n}_{n-1}(g_{k})\to
U_{\oO}C_{k}(n+m)\sma_{\Sigma_{n}} Y^{(n)}_{k}
\]
is an $h$-cofibration.  Each map above
and therefore each map
\[
Fil_{n-1}(C_{k},g_{k})\to Fil_{n}(C_{k},g_{k})
\]
(for $m=0$) and each map
\[
Fil_{n-1}(U_{\oO}C_{k},g_{k})(m)\to Fil_{n}(U_{\oO}C_{k},g_{k})(m)
\]
is $\Sigma_{m}$-equivariantly a retract of a relative cell complex
built out of cells of the form
\[
\Sigma_{m+}\sma X\to \Sigma_{m+}\sma Y
\]
where $X\to Y$ is a wedge of maps in $I$.  It follows that $C_{k}\to
C_{k+1}$ is a cofibration in $\Spectra$ and that $U_{\oO}C_{k+1}$ is a
$\Sigma$-free cell retract.

The case when $\Spectra$ is EKMM $\bL$-spectra and $\oO$ is an operad
of spaces is similar except that the inductive hypothesis on
$U_{\oO}C_{k}$ is replaced by the hypothesis that for each $m$,
$\oO(m)\sma S\to U_{\oO}C_{k}(m)$ is the retract of a relative cell
complex built out of cells of the form
\[
\Sigma_{m+}\sma X\to \Sigma_{m+}\sma Y
\]
where $X\to Y$ is a wedge of maps in $I$.
\end{proof}


\bibliographystyle{amsplain}\let\bibcomma\relax

\def\noopsort#1{}\def\MR#1{}\def\preprint{\ifx\bibcomma\undefined Preprint\else
  preprint\fi}
\providecommand{\bysame}{\leavevmode\hbox to3em{\hrulefill}\thinspace}
\providecommand{\MR}{\relax\ifhmode\unskip\space\fi MR }
\providecommand{\MRhref}[2]{%
  \href{http://www.ams.org/mathscinet-getitem?mr=#1}{#2}
}
\providecommand{\href}[2]{#2}

\end{document}